\documentclass{amsart}
\usepackage{amsmath}
\usepackage{amssymb}
\usepackage{amsthm}
\usepackage{graphicx}
\usepackage{hyperref}
\usepackage{mathrsfs}
\usepackage{mathtools}
\usepackage{pdfpages}
\usepackage[all,pdf]{xy}

\newtheorem{thm}{Theorem}[section]
\newtheorem{defn}[thm]{Definition}
\newtheorem{cor}[thm]{Corollary}

\newtheorem{lem}[thm]{Lemma}
\newtheorem{conj}[thm]{Conjecture}
\newtheorem*{rmk}{Remark}

\DeclareMathOperator{\R}{\mathbb{R}}
\DeclareMathOperator{\C}{\mathbb{C}}
\DeclareMathOperator{\FC}{\textup{Fun}\left(\Delta^1,\mathcal{C}\right)}
\DeclareMathOperator{\PP}{\mathbb{P}}
\DeclareMathOperator{\F}{\mathbb{F}_2}
\DeclareMathOperator{\Z}{\mathbb{Z}}
\DeclareMathOperator{\SC}{\mathbb{S}_{C_2}}
\DeclareMathOperator{\SR}{\mathbb{S}_{\mathbb{R}}}
\DeclareMathOperator{\A}{\mathcal{A}}
\DeclareMathOperator{\Ext}{\textup{Ext}}
\DeclareMathOperator{\Map}{\textup{Map}}
\DeclareMathOperator{\Be}{\textup{Be}^{C_2}}
\DeclareMathOperator{\Bec}{\widehat{\textup{Be}}_2^{C_2}}
\DeclareMathOperator{\Rcell}{\textup{SH}_{\textup{cell}}(\R)}

\title{The Borel and genuine $C_2$-equivariant Adams spectral sequences}
\author{Sihao Ma}

\begin{document}

\begin{abstract}
We find some relations between the classical Adams spectral sequences for stunted real projective spectra, the Borel $C_2$-equivariant Adams spectral sequence for the 2-completed sphere, and the genuine $C_2$-equivariant Adams spectral sequence for the 2-completed sphere. This allows us to understand the genuine $C_2$-equivariant Adams spectral sequence from the Borel Adams spectral sequences. We show that the Borel Adams spectral sequence is computable as a classical Adams spectral sequence.
\end{abstract}

\maketitle

\tableofcontents

\section{Introduction}\label{sec-intro}
In this article, everything is assumed to be $2$-completed implicitly. 

In $Sp^{C_2}$, there are two Adams spectral sequences converging to the $C_2$-equivariant stable homotopy groups of spheres: the genuine $C_2$-equivariant Adams spectral sequence (\cite{hk})
\[\Ext_{\A^{C_2}_{*,*}}^{s,t,w}\left(\left(H_{C_2}\right)_{*,*},\left(H_{C_2}\right)_{*,*}\right)\implies\pi_{t-s,w}^{C_2}\SC,\]
and the Borel $C_2$-equivariant Adams spectral sequence (Greenlees' works \cite{gre1,gre2,gre3})
\[\Ext_{\A^{h}_{*,*}}^{s,t,w}\left(\left(H_{C_2}\right)^h_{*,*},\left(H_{C_2}\right)^h_{*,*}\right)\implies\pi_{t-s,w}^{C_2}\left(\SC\right)^h.\]
By the Segal conjecture for $C_2$ \cite{ldma}, the map
\[\SC\to\left(\SC\right)^h\]
is a ($2$-adic) equivalence in $Sp^{C_2}$. However, the two spectral sequences above are not isomorphic.

Let $a\in\pi_{-1,-1}^{C_2}\SC$ be represented by the inclusion $S^0\to S^{\sigma}$. Let $Ca$ be the cofiber of $a$. The $E_2$-page of the Borel Adams spectral sequence for $\SC$ can be computed through the algebraic $a$-Bockstein spectral sequence. There is also a topological $a$-Bockstein spectral sequence converging to $\pi_{*,*}^{C_2}\SC$. Note that by definition the $a$-Bockstein spectral sequence only converges to the $a$-complete (Borel complete) sphere. Since we are working in the $2$-complete category, this is equivalent to $\SC$. Therefore, we have the following square:
\begin{equation}\label{square1}
\begin{aligned}
\xymatrix@C=1cm{
 & \Ext_{\A_{*,*}^h}\left((H_{C_2}\wedge Ca)_{*,*}^h\right)[a] \ar@{=>}[rd]|{\textbf{algebraic }a-\textbf{Bockstein SS}} \ar@{=>}[ld]|{\textbf{Borel Adams SS}} & \\
\pi_{*,*}^{C_2}\left(Ca\right)[a] \ar@{=>}[rd]|{\textbf{topological }a-\textbf{Bockstein SS}} & & \Ext_{\A_{*,*}^h}\left((H_{C_2})^h_{*,*}\right) \ar@{=>}[ld]|{\textbf{Borel Adams SS}}\\
 & \pi_{*,*}^{C_2}\SC. & \\
}
\end{aligned}
\end{equation}

Section \ref{sec-background} is a background section, which aims to provide some knowledge of the $C_2$-equivariant and $\R$-motivic stable homotopy theory that will be needed throughout this paper.

In Section \ref{sec-borel}, we will prove that the Borel Adams spectral sequence for $\SC$ is isomorphic to the limit of a sequence of classical Adams spectral sequences for stunted real projective spectra, which converges to $\pi_*\PP_{-\infty}^{-w-1}$ (c.f. Theorem \ref{borandlim}), where $\PP_{-\infty}^{-w-1}=\varprojlim\limits_k \PP_{-k}^{-w-1}$. The fiber sequence
\[\PP_{-\infty}^{-j-2}\to\PP_{-\infty}^{-j-1}\to S^{-j-1}\]
and the filtration on the limit of homology groups give us the topological and algebraic Atiyah-Hirzebruch spectral sequences, which will be proved to be isomorphic to the topological and algebraic $a$-Bockstein spectral sequences (c.f. Theorem \ref{topabss} and Theorem \ref{rhobss}). In conclusion, there is a square in the classical setting isomorphic to the one above:
\begin{equation}\label{square2}
\begin{aligned}
\xymatrix@!C@=1cm{
 & \bigoplus\Ext_{\A_*}(\F) \ar@{=>}[rd]|{\textbf{algebraic Atiyah-Hirzebruch SS}} \ar@{=>}[ld]|{\textbf{Adams SS}} & \\
\bigoplus\pi_*\mathbb{S} \ar@{=>}[rd]|{\textbf{topological Atiyah-Hirzebruch SS}} & & \Ext_{\A_*}(\varprojlim\limits_kH_*\PP_{-k}^{-w-1}) \ar@{=>}[ld]|{\textbf{Adams SS}}\\
 & \pi_*\PP_{-\infty}^{-w-1}. & \\
}
\end{aligned}
\end{equation}
Note that all these $4$ spectral sequences converge. The convergence of the Adams spectral sequence for the $2$-complete sphere is well known (see \cite{adams}). The convergence of the Adams spectral sequence of the inverse limit is discussed in \cite{lin}. The convergence of the topological Atiyah-Hirzebruch spectral sequence is discussed in \cite[Cor.~11.4]{gm}. Similar considerations imply the convergence of the algebraic Atiyah-Hirzebruch spectral sequence.

In Section \ref{sec-limit}, we will focus on the limit of the Adams spectral sequences. It is hard to compute directly, as more and more information about higher stable stems is involved as the dimension of the bottom cell gets lower and lower. To solve this problem, we will show that it is also isomorphic to the direct sum of at most two classical Adams spectral sequences for some bounded below spectra of finite type, which can be computed (relatively) easily (c.f. Theorem \ref{posw} and Theorem \ref{negw}).

In Section \ref{sec-lemma}, we will prove a $4\times 4$ lemma which generalizes \cite[Lem.~9.3.2]{am17} (c.f. Lemma \ref{lem6}). The results in this Section will be applied in Section \ref{sec-genuine}.

In Section \ref{sec-genuine}, we will discuss the relationship between the Borel and genuine $C_2$-equivariant Adams spectral sequences for $\SC$. We will prove that the $E_2$-term of the Borel one can be obtained from the $E_2$-term of the genuine one by shifting the degree of some elements and cancelling some genuine $d_2$ differentials (c.f. Theorem \ref{borext}). We will also show that their differentials can be deduced from each other up to indeterminacy (c.f. Theorem \ref{shortd} and Theorem \ref{diffnc}).

In Section \ref{sec-names}, we will discuss the relations of two different naming systems of elements appeared in both the genuine $\Ext$ and the Borel $\Ext$, one from the $a$-Bockstein spectral sequence, the other from the algebraic Atiyah-Hirzebruch spectral sequence (c.f. Section \ref{sec-limit}). For some elements, we will show that we can deduce one of its names from the other (c.f. Theorem \ref{pcnames} and Theorem \ref{ncnames}), while we also leave a case open (c.f. Conjecture \ref{conjnc}).

%In Appendix \ref{sec-charts}, we include the charts of $\Ext_{\A_{*,*}^h}^{s,t,w}\left((H_{C_2})^h_{*,*}\right)$ for $0\le t-s\le 30$ and $-2\le t-s-w\le13$. The data are obtained from a computer program \cite{prog} which computes the algebraic Atiyah-Hirzebruch spectral sequence in Section \ref{sec-limit}.

\textbf{Acknowledgement:} The author would like to thank Mark Behrens for encouragement to think about this problem and many helpful conversations. The author would like to thank William Balderrama for valuable comments on previous drafts of this article. The author would like to thank Bert Guillou and Dan Isaksen for sharing their computations. The author would also like to thank the annonymous referee for many detailed suggestions.
\vskip.1in
\noindent\textbf{Notation:}
\begin{itemize}
\item $H_{\textup{cl}}$ is the classical Eilenberg-MacLane spectrum associated to $\F$.
\item $H_{C_2}$ is the $C_2$-equivariant Eilenberg-MacLane spectrum associated to the constant Mackey functor $\underline{\F}$.
\item $H_{\R}$ is the $\R$-motivic Eilenberg-MacLane spectrum associated to $\F$.
\item $\overline{H_{\textup{cl}}}$ is the fiber of the unit map $\mathbb{S}\to H_{\textup{cl}}$.
\item $\overline{H_{C_2}}$ is the fiber of the unit map $\SC\to H_{C_2}$.
\item $\overline{H_{\R}}$ is the fiber of the unit map $\SR\to H_{\R}$.
\item For a $C_2$-equivariant spectrum $X$, $X^h$ is its homotopy completion, $X^{\Phi}$ is its geometric localization, and $X^t$ is its Tate spectrum.
\item For a (2-complete) cellular $\R$-motivic spectrum $X$, $X^h$ is $X^{\wedge}_{\rho}[\tau^{-1}]$, $X^{\Phi}$ is $X[\rho^{-1}]$, and $X^t$ is $X^{\wedge}_{\rho}[\tau^{-1}][\rho^{-1}]$. $X^{C_2}$ is the pullback of
\[\xymatrix{ & X^{\Phi} \ar[d] \\ X^h \ar[r] & X^t. \\}\]
$X^{\Theta}$ is the cofiber of the map $X^{C_2}\to X^h$. $X^{\Psi}$ is the cofiber of the map $X\to X^h$ (c.f. Section \ref{subsec-rmot}).
\item For a $\R$-motivic spectrum $X$, $X_{*,*}$ are its bigraded homotopy groups $\pi_{*,*}^{\R}(X)$.
\item For a $C_2$-equivariant spectrum $X$, $X_{*,*}$ are its bigraded homotopy groups $\pi_{*,*}^{C_2}(X)$. The index is taken such that $\pi_{i,j}^{C_2}(X)=\pi_{(i-j)+j\sigma}^{C_2}(X)$, where $\sigma$ is the sign representation.
\item $\A_*$ is the classical dual Steenrod algebra.
\item $\A^h_{*,*}$ is the Borel $C_2$-equivariant dual Steenrod algebra.
\item $\A^{C_2}_{*,*}$ is the genuine $C_2$-equivariant dual Steenrod algebra.
\item $\A_{\R}$ is the $\R$-motivic Steenrod algebra. $\A^{\R}_{*,*}$ is its dual.
\item For a Hopf algebroid $(A,\Gamma)$, and a left $\Gamma$-comodule $M$, $C_{\Gamma}(A,M)$ is the reduced cobar complex $\{\Gamma\otimes_A\overline{\Gamma}^{\otimes s}\otimes_A M\}$. If there is no ambiguity, it will be abbreviated by $C_{\Gamma}(M)$, and $\Ext_{\Gamma}(A,M)$ will be abbreviated by $\Ext_{\Gamma}(M)$.
\item $\Ext_{\C}$ is the cohomology of the $\C$-motivic dual Steenrod algebra.
\item $(s,t,w)$ grading is used when $\Ext_{\A^{C_2}_{*,*}}^{*,*,*}$ and $\Ext_{\A^{\R}_{*,*}}^{*,*,*}$ are considered, where $s$ denotes the filtration degree, $t$ denotes the total degree, and $w$ is the motivic weight. Note that this is different from the $(s',f',w')$ grading, where $s'=t-s$, $f'=s$, and $w'=w$.
\item $Y/X$ is the cofiber of the map $X\to Y$.
\item $\F[x]/(x^\infty)$ is the cokernel of the inclusion map $\F[x]\to\F[x^\pm]$.
\end{itemize}

\section{Background}\label{sec-background}
\subsection{$C_2$-equivariant Adams spectral sequences}\hfill

Recall from \cite{hk} that
\[\left(H_{C_2}\right)_{*,*}=\F[u,a]\oplus\left(\bigoplus_{j\ge0}\frac{\F[u]}{\left(u^\infty\right)}\left\{\frac{\gamma}{a^j}\right\}\right),\]
where $|u|=(0,-1)$, $|a|=(-1,-1)$, and $|\gamma/u|=(0,2)$. Note that $a$ is the Hurewicz image of 
\[a\in\pi_{-1,-1}^{C_2}\SC,\]
which is realised by the inclusion $S^0\to S^{\sigma}$. The genuine $C_2$-equivariant dual Steenrod algebra is (\cite{hk})
\[\A^{C_2}_{*,*}=\left(H_{C_2}\wedge H_{C_2}\right)_{*,*}=\frac{\left(H_{C_2}\right)_{*,*}[\xi_i,\tau_j:i\ge1,j\ge0]}{\left(\tau_j^2=a\tau_{j+1}+u\xi_{j+1}+a\tau_0\xi_{j+1}\right)},\]
where $|\xi_i|=(2^{i+1}-2,2^i-1)$, and $|\tau_j|=(2^{j+1}-1,2^j-1)$. Note that $\A^{C_2}_{*,*}$ is a free $\left(H_{C_2}\right)_{*,*}$-module, and $\left(\left(H_{C_2}\right)_{*,*},\A^{C_2}_{*,*}\right)$ is a Hopf algebroid. The genuine $C_2$-equivariant Adams spectral sequence for a finite spectrum $X$ can be constructed through the $H_{C_2}$-based Adams resolution
\begin{equation}\label{genadamsres}
\begin{aligned}
\xymatrix{
X \ar[d] & \overline{H_{C_2}}\wedge X \ar[l] \ar[d] & \overline{H_{C_2}}^{\wedge2}\wedge X \ar[l]\ar[d] & \cdots \ar[l] \\
H_{C_2}\wedge X & H_{C_2}\wedge\overline{H_{C_2}}\wedge X & H_{C_2}\wedge\overline{H_{C_2}}^{\wedge2}\wedge X. & \\
}
\end{aligned}
\end{equation}
Its $E_1$-page is
\[E_1^{s,t,w}=\pi_{t-s,w}^{C_2}H_{C_2}\wedge\overline{H_{C_2}}^{\wedge s}\wedge X=C_{\A^{C_2}_{*,*}}^{s,t,w}\left(\left(H_{C_2}\wedge X\right)_{*,*}\right),\]
and its $E_2$-page is
\[E_2^{s,t,w}=\Ext_{\A^{C_2}_{*,*}}^{s,t,w}\left(\left(H_{C_2}\wedge X\right)_{*,*}\right).\]

Consider the homotopy completion functor in $Sp^{C_2}$:
\[X^h=\Map\left(\left(EC_2\right)_+,X\right).\]
In \cite{hk}, Hu-Kriz found that
\[\left(H_{C_2}\right)^h_{*,*}=\F[u^{\pm},a].\]
To compute the Borel $C_2$-equivariant dual Steenrod algebra, we need to work in the category $\mathscr{M}$ of bigraded $\Z[a]$-modules complete with respect to the topology associated with the principal ideal $(a)$, and continuous homomorphisms. In \cite{hk}, Hu-Kriz also computed that
\begin{align*}
\A^h_{*,*}=\left(H_{C_2}\wedge H_{C_2}\right)^h_{*,*} &\cong\left(H_{C_2}\right)^h_{*,*}[\zeta_i:i\ge1]^{\wedge}_a\\
&\cong\frac{\left(H_{C_2}\right)^h_{*,*}[\xi_i,\tau_j:i\ge1,j\ge0]^{\wedge}_a}{\left(\tau_j^2=a\tau_{j+1}+u\xi_{j+1}+a\tau_0\xi_{j+1}\right)},
\end{align*}
where $\zeta_i$ is the image of $\zeta_i\in\A_*$ under the homotopy completion map. The image of $\zeta_i$ under the isomorphism can be computed inductively using the formula
\[\zeta_{i+1}=a^{-1}\left(\left(u+a\tau_0\right)\zeta_i^2+u^{2^i}\xi_i+a^{2^{i+1}}\xi_{i+1}\right).\]
It can be seen that the image of $\zeta_{i+1}$ has the leading term $u^{2^i-1}\tau_i$ under the $(a)$-adic filtration. Note that $\left(\left(H_{C_2}\right)^h_{*,*},\A^h_{*,*}\right)$ is only a Hopf algebroid in $\mathscr{M}$. Greenlees \cite{gre1,gre2,gre3} constructed the Borel $C_2$-equivariant Adams spectral sequence for a finite spectrum $X$ by applying the homotopy completion functor to the resolution (\ref{genadamsres}) to get
\begin{equation*}
\begin{aligned}
\xymatrix{
X^h \ar[d] & \left(\overline{H_{C_2}}\wedge X\right)^h \ar[l] \ar[d] & \left(\overline{H_{C_2}}^{\wedge2}\wedge X\right)^h \ar[l]\ar[d] & \cdots \ar[l] \\
\left(H_{C_2}\wedge X\right)^h & \left(H_{C_2}\wedge\overline{H_{C_2}}\wedge X\right)^h & \left(H_{C_2}\wedge\overline{H_{C_2}}^{\wedge2}\wedge X\right)^h. & \\
}
\end{aligned}
\end{equation*}
Its $E_1$-page is
\[E_1^{s,t,w}=\pi_{t-s,w}^{C_2}\left(H_{C_2}\wedge\overline{H_{C_2}}^{\wedge s}\wedge X\right)^h=C_{\A^h_{*,*}}^{s,t,w}\left(\left(H_{C_2}\wedge X\right)^h_{*,*}\right),\]
and its $E_2$-page is
\[E_2^{s,t,w}=\Ext_{\A^h_{*,*}}^{s,t,w}\left(\left(H_{C_2}\wedge X\right)^h_{*,*}\right).\]

\subsection{$\mathbb{R}$-motivic stable homotopy theory}\label{subsec-rmot}\hfill

Let $\textup{SH}(\R)$ be the $\R$-motivic stable homotopy category. Let $\Rcell$ be the localizing subcategory of $\textup{SH}(\R)$ generated by the motivic spheres $\left\{S_{\R}^{p,q}\right\}$.

In \cite{voeH}, Voevodsky computed that
\[\left(H_{\R}\right)_{*,*}=\F[\tau,\rho],\]
where $|\tau|=(0,-1)$, and $|\rho|=(-1,-1)$. Note that $\rho$ is the Hurewicz image of 
\[\rho\in\pi_{-1,-1}^{\R}\SR,\]
which is realised by the inclusion $S^{0,0}_{\R}=\{\pm1\}\to\mathbb{G}_m=S^{1,1}_{\R}$. Voevodsky also computed the $\R$-motivic Steenrod algebra (\cite{voeA})
\[\A^{\R}_{*,*}=\left(H_{\R}\wedge H_{\R}\right)_{*,*}=\frac{\left(H_{\R}\right)_{*,*}[\xi_i,\tau_j:i\ge1,j\ge0]}{\left(\tau_j^2=\rho\tau_{j+1}+\tau\xi_{j+1}+\rho\tau_0\xi_{j+1}\right)},\]
where $|\xi_i|=(2^{i+1}-2,2^i-1)$, and $|\tau_j|=(2^{j+1}-1,2^j-1)$.

The $\R$-motivic Adams spectral sequence for a cellular spectrum of finite type $X$ (\cite{mor} \cite{di} \cite{hko}) is induced by the $H_{\R}$-based Adams resolution 
\begin{equation*}
\begin{aligned}
\xymatrix{
X \ar[d] & \overline{H_{\R}}\wedge X \ar[l] \ar[d] & \overline{H_{\R}}^{\wedge2}\wedge X \ar[l]\ar[d] & \cdots \ar[l] \\
H_{\R}\wedge X & H_{\R}\wedge\overline{H_{\R}}\wedge X & H_{\R}\wedge\overline{H_{\R}}^{\wedge2}\wedge X. & \\
}
\end{aligned}
\end{equation*}
Its $E_1$-page is
\[E_1^{s,t,w}=\pi_{t-s,w}^{\R}H_{\R}\wedge\overline{H_{\R}}^{\wedge s}\wedge X=C_{\A^{\R}_{*,*}}^{s,t,w}\left(\left(H_{\R}\wedge X\right)_{*,*}\right),\]
and its $E_2$-page is
\[E_2^{s,t,w}=\Ext_{\A^{\R}_{*,*}}^{s,t,w}\left(\left(H_{\R}\wedge X\right)_{*,*}\right).\]

By \cite[Const.~1.2.2.6]{ha}, we can express the $E_r$-terms of the $\R$-motivic Adams spectral sequence for $\SR$ as
\[E_r^{s,t,w}\cong\textup{im}\left(\pi^{\R}_{t-s,w}\left(\overline{H_{\R}}^{\wedge s}\middle/\overline{H_{\R}}^{\wedge (s+r)}\right)\to\pi^{\R}_{t-s,w}\left(\overline{H_{\R}}^{\wedge (s+1-r)}\middle/\overline{H_{\R}}^{\wedge (s+1)}\right)\right),\]
where the map
\[\left.\overline{H_{\R}}^{\wedge a}\middle/\overline{H_{\R}}^{\wedge b}\right.\to\left.\overline{H_{\R}}^{\wedge c}\middle/\overline{H_{\R}}^{\wedge d}\right. \quad(a\ge c,\ b\ge d)\]
is obtained by the iteration of the map
\[\overline{H_{\R}}\to\SR.\]
The $E_r$-terms can be viewed as the subquotient of the $E_1$-terms by the isomorphism
\[E_r^{s,t,w}\cong\frac{\textup{im}\left(\pi^{\R}_{t-s,w}\left(\overline{H_{\R}}^{\wedge s}\middle/\overline{H_{\R}}^{\wedge (s+r)}\right)\to\pi^{\R}_{t-s,w}\left(\overline{H_{\R}}^{\wedge s}\middle/\overline{H_{\R}}^{\wedge (s+1)}\right)\right)}{\textup{ker}\left(\pi^{\R}_{t-s,w}\left(\overline{H_{\R}}^{\wedge s}\middle/\overline{H_{\R}}^{\wedge (s+1)}\right)\to\pi^{\R}_{t-s,w}\left(\overline{H_{\R}}^{\wedge (s+1-r)}\middle/\overline{H_{\R}}^{\wedge (s+1)}\right)\right)}.\]
In a similar way, they can also be viewed as the subquotient of $\pi^{\R}_{t-s,w}\left(\overline{H_{\R}}^{\wedge a}\middle/\overline{H_{\R}}^{\wedge b}\right)$ for any $s+1-r\le a\le s$ and $s+1\le b\le s+r$. The differential $d_r$ can be obtained through
\[
\xymatrix{
\pi^{\R}_{t-s,w}\left(\overline{H_{\R}}^{\wedge s}\middle/\overline{H_{\R}}^{\wedge (s+r)}\right) \ar[r]\ar[d] & \pi^{\R}_{t-s,w}\left(\overline{H_{\R}}^{\wedge (s+1-r)}\middle/\overline{H_{\R}}^{\wedge (s+1)}\right) \ar[d] \\
\pi^{\R}_{t-s-1,w}\left(\overline{H_{\R}}^{\wedge (s+r)}\middle/\overline{H_{\R}}^{\wedge (s+2r)}\right) \ar[r] & \pi^{\R}_{t-s-1,w}\left(\overline{H_{\R}}^{\wedge (s+1)}\middle/\overline{H_{\R}}^{\wedge (s+1+r)}\right). \\
}
\]

A bridge connecting the $\R$-motivic stable homotopy theory and the $C_2$-equivariant homotopy theory is the $C_2$-Betti realization functor constructed by Morel-Voevodsky \cite{mv}:
\[\Be:\textup{SH}(\R)\to Sp^{C_2}.\]
In \cite{ho}, Heller-Ormsby proved that 
\[\Be(H_{\R})=H_{C_2}.\]
In their homotopy groups, it turns out that $\Be(\tau)=u$ and $\Be(\rho)=a$.

In \cite{bac18}, Bachmann proved that for $X\in\textup{SH}(\R)$, the $C_2$-equivariant Betti realization induces an isomorphism
\begin{equation}\label{isoh}
\pi_{*,*}^{\R}X[\rho^{-1}]\xrightarrow{\cong}\pi_{*,*}^{C_2}\Be(X)^{\Phi},
\end{equation}
where $\Be(X)^{\Phi}$ is the geometric localization of $\Be(X)$. In \cite{bs}, Behrens-Shah defined the $2$-complete $C_2$-Betti realization functor
\[\Bec:\Rcell^{\wedge}_2\to(Sp^{C_2})^{\wedge}_2;\quad\Bec(X)=\Be(X)^{\wedge}_2.\]
For $X\in\Rcell^{\wedge}_2$, they proved another isomorphism induced by $\Bec$
\begin{equation}\label{isophi}
\pi_{*,*}^{\R}X^{\wedge}_{\rho}[\tau^{-1}]\xrightarrow{\cong}\pi_{*,*}^{C_2}\Bec(X)^h,
\end{equation}
and hence, there is also
\begin{equation}\label{isot}
\pi_{*,*}^{\R}X^{\wedge}_{\rho}[\tau^{-1}][\rho^{-1}]\xrightarrow{\cong}\pi_{*,*}^{C_2}\Bec(X)^t.
\end{equation}
Here $X^{\wedge}_{\rho}[\tau^{-1}]$ is defined as follows: For each $i\ge1$, $C(\rho^i)^{\wedge}_2$ admits a $\tau^j$-self map. Let $C(\rho^i)^{\wedge}_2[\tau^{-1}]$ be the telescope, and we define
\[X^{\wedge}_{\rho}[\tau^{-1}]:=\varprojlim_i X\wedge C(\rho^i)^{\wedge}_2[\tau^{-1}].\]
These three isomorphisms ((\ref{isoh}), (\ref{isophi}), (\ref{isot})) explain our notations $X^h$, $X^{\Phi}$, and $X^t$ (c.f. Section \ref{sec-intro}). Applying the isotropy separation square to $\Bec(X)$, we also have the following isomorphism for $X\in\Rcell^{\wedge}_2$:
\[\pi_{*,*}^{\R}X^{C_2}\xrightarrow{\cong}\pi_{*,*}^{C_2}\Bec(X).\]

In \cite{bs}, Behrens-Shah also proved that $\Bec$ has a fully faithful right adjoint $\textup{Cell Sing}^{C_2}$, which allows us to regard $\left(Sp^{C_2}\right)^{\wedge}_2$ as a localization of $\Rcell^{\wedge}_2$, and our $X^{C_2}$ is equivalent to $\textup{Cell Sing}^{C_2}\Bec(X)$. Moreover, they showed that this pair of adjoint functors satisfies the projection formula
\[\textup{Cell Sing}^{C_2}(A)\wedge B\xrightarrow{\simeq}\textup{Cell Sing}^{C_2}\left(A\wedge\Bec(B)\right)\]
for all $A\in\left(Sp^{C_2}\right)^{\wedge}_2$ and $B\in\Rcell^{\wedge}_2$. In particular, when $A\simeq\SC$, we have
\[B^{C_2}\simeq B\wedge \left(\SR\right)^{C_2}.\]
Therefore, $(-)^{C_2}$ is a smashing localization.

Now we define the $\R$-motivic spectra $X^{\Theta}$ and $X^{\Psi}$ for $X\in\Rcell^{\wedge}_2$ (c.f. Section \ref{sec-intro}) for later use.

\begin{defn}
For $X\in\Rcell^{\wedge}_2$, let
\[X^{\Theta}=\left.X^h\middle/X^{C_2}\right.,\]
and
\[X^{\Psi}=\left.X^h\middle/X\right..\]
\end{defn}

By the Segal conjecture for $C_2$ \cite{ldma}, the map
\[\SC\to\left(\SC\right)^h\]
is an equivalence in $\left(Sp^{C_2}\right)^{\wedge}_2$. Thus there is an equivalence 
\[\left(\SR\right)^{C_2}\xrightarrow{\simeq}\left(\SR\right)^h\]
in $\Rcell^{\wedge}_2$. Then we have
\[\left(\SR\right)^{\Theta}\simeq\left.\left(\SR\right)^h\middle/\left(\SR\right)^{C_2}\right.\simeq*,\]
and 
\[\left(\SR\right)^{\Psi}\simeq\left.\left(\SR\right)^h\middle/\SR\right.\simeq\left.\left(\SR\right)^{C_2}\middle/\SR\right..\]
Moreover,
\[\left.X^{C_2}\middle/X\right.\simeq\left.\left(X\wedge\left(\SR\right)^{C_2}\right)\middle/\left(X\wedge\SR\right)\right.\simeq X\wedge\left(\SR\right)^{\Psi}.\]
However, $X^{C_2}$ may not be equivalent to $X^h$ in general. For example, when $X\simeq H_{\R}$, $\left(H_{\R}\right)^{C_2}$ and $\left(H_{\R}\right)^h$ have different $\R$-motivic homotopy groups, and hence, are not equivalent. As a consequence, $X^{\Psi}$ may not be equivalent to $X\wedge\left(\SR\right)^{\Psi}$ in general.

Here we list the $\R$-motivic homotopy groups of some spectra related to $H_{\R}$, which can be easily computed from their definition (where $\F[\tau]/\left(\tau^\infty\right)$ denotes the cokernel of the inclusion map $\F[\tau]\to\F[\tau^\pm]$):
\begin{equation*}
\begin{split}
&\left(H_{\R}\right)^{C_2}_{*,*}=\F[\tau,\rho]\oplus\left(\bigoplus_{j\ge0}\frac{\F[\tau]}{\left(\tau^\infty\right)}\left\{\frac{\gamma}{\rho^j}\right\}\right),\\
&\left(H_{\R}\right)^h_{*,*}=\F[\tau^{\pm},\rho],\\
&\left(H_{\R}\right)^{\Phi}_{*,*}=\F[\tau,\rho^{\pm}],\\
&\left(H_{\R}\right)^t_{*,*}=\F[\tau^{\pm},\rho^{\pm}],\\
&\left(H_{\R}\right)^{\Theta}_{*,*}=\frac{\F[\tau,\rho^{\pm}]}{\left(\tau^{\infty}\right)},\\
&\left(H_{\R}\right)^{\Psi}_{*,*}=\frac{\F[\tau,\rho]}{\left(\tau^{\infty}\right)},\\
&\left(H_{\R}\wedge\left(\SR\right)^{\Psi}\right)_{*,*}=\bigoplus_{j\ge0}\frac{\F[\tau]}{\left(\tau^\infty\right)}\left\{\frac{\gamma}{\rho^j}\right\}.
\end{split}
\end{equation*}
Note that $\F[\tau,\rho]\subset\left(H_{\R}\right)^{C_2}_{*,*}$ is called the positive cone, $\bigoplus_{j\ge0}\F[\tau]/\left(\tau^\infty\right)\left\{\gamma/\rho^j\right\}\subset\left(H_{\R}\right)^{C_2}_{*,*}$ is called the negative cone, where $\gamma/\tau$ is the image of $1/\left(\tau\rho\right)\in\left(H_{\R}\right)^{\Theta}_{1,2}$ under the connecting map $\left(H_{\R}\right)^\Theta\to\Sigma^{1,0}\left(H_{\R}\right)^{C_2}$. In this way, we have constructed the $\R$-motivic spectrum $H_{\R}\wedge\left(\SR\right)^{\Psi}$ whose $\R$-motivic homotopy groups are isomorphic to the negative cone.

\section{The Borel $C_2$-equivariant and classical Adams spectral sequences}\label{sec-borel}
In this section, we will prove that the squares \eqref{square1} and \eqref{square2} are isomorphic.

We will begin with a lemma, which will be used in Section \ref{sec-limit}.

\begin{lem}\label{pands}
There is an isomorphism
\[\pi_{i,j}^{C_2}\SC\cong\pi_{i-j-1}\PP_{-\infty}^{-j-1}.\]
\end{lem}

\begin{proof}
By \cite[Prop.~7.1]{bs}, we have
\[\Map\left(S^{j\sigma},\SC\right)^{C_2}\simeq\Map\left(\PP_j^{\infty}, \mathbb{S}\right).\]
By \cite[Thm.~2.14]{bruner}, we have
\[\Map\left(\PP_j^{\infty}, \mathbb{S}\right)\simeq\Map\left(\varinjlim_k\PP_j^k, \mathbb S\right)\simeq \varprojlim_k\Sigma\PP_{-k-1}^{-j-1}\simeq \Sigma\PP_{-\infty}^{-j-1}.\]
Therefore, we have
\[\pi_{i,j}^{C_2}\SC\cong\pi_{i-j}\Map\left(S^{j\sigma},\SC\right)^{C_2}\cong\pi_{i-j-1}\PP_{-\infty}^{-j-1}.\]
\end{proof}

For the lower left sides of the squares, we have

\begin{thm}\label{topabss}
The topological $a$-Bockstein spectral sequence for $\pi_{i,j}^{C_2}\SC$
\[
\bigoplus_{n=0}^{\infty}\pi_{i+n,j+n}^{C_2}Ca\Rightarrow \pi_{i,j}^{C_2}\SC
\]
is isomorphic to the topological Atiyah-Hirzebruch spectral sequence for $\pi_{i-j-1}\PP^{-j-1}_{-\infty}$, the homotopy groups of the stunted real projective spectra
\[
\bigoplus_{n=0}^{\infty}\pi_{i-j-1}S^{-j-1-n}\Rightarrow \pi_{i-j-1}\PP_{-\infty}^{-j-1}
\]
from their $E_1$-pages.
\end{thm}

\begin{proof}
Note that the topological $a$-Bockstein spectral sequence is induced by the $C_2$-equivariant homotopy groups of the fiber sequence
\begin{equation}\label{cafibseq}
\xymatrix{
\Sigma^{-\sigma}\SC \ar[r]^(0.58){a} & \SC \ar[r] & Ca \\
}.
\end{equation}
For $X\in Sp^{C_2}$, we have an isomorphism
\[\pi_{i,j}^{C_2}X\cong\pi_{i-j}\Map\left(S^{j\sigma},X\right)^{C_2}.\]
Thus, the $C_2$-equivariant homotopy groups of (\ref{cafibseq}) are isomorphic to the classical homotopy groups of the fiber sequence
\begin{equation}\label{mapfibseq}
\xymatrix{
\Map\left(S^{j\sigma},\Sigma^{-\sigma}\SC\right)^{C_2} \ar[r]^(0.54){a} & \Map\left(S^{j\sigma},\SC\right)^{C_2} \ar[r] & \Map\left(S^{j\sigma},Ca\right)^{C_2} \\
}.
\end{equation}
By the proof of Lemma \ref{pands}, we have
\[\Map\left(S^{j\sigma},\Sigma^{-\sigma}\SC\right)^{C_2}\simeq\Sigma\PP_{-\infty}^{-j-2},\]
and
\[\Map\left(S^{j\sigma},\SC\right)^{C_2}\simeq\Sigma\PP_{-\infty}^{-j-1}.\]
For the third term, since
\[Ca\simeq\Sigma^{1-\sigma}\textup{fib}\left(S^0\to S^\sigma\right)\simeq\Sigma^{1-\sigma}\Sigma^{\infty}\left(C_2\right)_+,\]
we have
\[\Map\left(S^{j\sigma},Ca\right)^{C_2}\simeq\Map\left(S^{(j+1)\sigma-1},\Sigma^{\infty}\left(C_2\right)_+\right)^{C_2}\simeq S^{-j}.\]
Therefore, the fiber sequence (\ref{mapfibseq}) is equivalent to the fiber sequence
\[\Sigma\PP_{-\infty}^{-j-2}\to\Sigma\PP_{-\infty}^{-j-1}\to\Sigma S^{-j-1},\]
which induces the topological Atiyah-Hirzebruch spectral sequence for $\Sigma\PP_{-\infty}^{-j-1}$ (and hence $\PP_{-\infty}^{-j-1}$). Then the result follows.
\end{proof}

Then consider the lower right sides of the squares.

\begin{thm}\label{borandlim}
The Borel Adams spectral sequence for the $C_2$-equivariant sphere
\[
\bigoplus_{n=0}^{\infty}\pi_{i,j}^{C_2}\left(H_{C_2}\wedge\overline{H_{C_2}}^{\wedge n}\right)^h \Rightarrow \pi_{i,j}^{C_2}\SC
\]
is isomorphic to the limit of classical Adams spectral sequences for the stunted projective spectra
\[
\bigoplus_{n=0}^{\infty}\pi_{i-j-1}\varprojlim\limits_k\left(\PP_{-k}^{-j-1}\wedge H_{\textup{cl}}\wedge\overline{H_{\textup{cl}}}^{\wedge n}\right) \Rightarrow \pi_{i-j-1}\PP_{-\infty}^{-j-1}
\]
from their $E_1$-pages.
\end{thm}

\begin{proof}
Since $C_2$ acts trivially on $\SC$ and $H_{C_2}$, its action on $\overline{H_{C_2}}$ is trivial as well. Then we have the isomorphism
\begin{align*}
\pi_{i,j}^{C_2}\Map\left((EC_2)_+,\overline{H_{C_2}}^{\wedge n}\right) & \cong \pi_{i-j}\Map\left((EC_2)_+\wedge S^{j\sigma}, \overline{H_{C_2}}^{\wedge n}\right)^{C_2} \\
 & \cong \pi_{i-j}\Map\left((EC_2)_+\wedge_{C_2} S^{j\sigma},\overline{H_{\textup{cl}}}^{\wedge n}\right) \\
 & \cong \pi_{i-j}\Map\left(\PP_j^{\infty}, \overline{H_{\textup{cl}}}^{\wedge n}\right) \\
 & \cong \pi_{i-j}\Map\left(\varinjlim\limits_k\PP_j^{k-1}, \overline{H_{\textup{cl}}}^{\wedge n}\right) \\
 & \cong \pi_{i-j}\varprojlim\limits_k\Map\left(\PP_j^{k-1}, \overline{H_{\textup{cl}}}^{\wedge n}\right) \\
 & \cong \pi_{i-j-1}\varprojlim\limits_k\left(\PP_{-k}^{-j-1}\wedge\overline{H_{\textup{cl}}}^{\wedge n}\right),
\end{align*}
and a similar isomorphism for $H_{C_2}\wedge\overline{H_{C_2}}^{\wedge n}$, which are compatible in the sense that there is a commutative diagram
\[
\xymatrix{
\pi_{i,j}^{C_2}\Map\left((EC_2)_+,\overline{H_{C_2}}^{\wedge n}\right) \ar[r]^{\cong} \ar[d] & \pi_{i-j-1}\varprojlim\limits_k\left(\PP_{-k}^{-j-1}\wedge\overline{H_{\textup{cl}}}^{\wedge n}\right) \ar[d] \\
\pi_{i,j}^{C_2}\Map\left((EC_2)_+,H_{C_2}\wedge\overline{H_{C_2}}^{\wedge n}\right) \ar[r]^{\cong} & \pi_{i-j-1}\varprojlim\limits_k\left(\PP_{-k}^{-j-1}\wedge H_{\textup{cl}}\wedge\overline{H_{\textup{cl}}}^{\wedge n}\right). \\
}
\]
Therefore, there is an isomorphism between the homotopy groups of the two Adams towers, and hence an isomorphism between the two spectral sequences.
\end{proof}

For the upper left sides of the squares, by \cite{bw}, we have
\[Ca\simeq\Sigma^{\sigma-1}Ca,\]
so $\pi_{i,j}^{C_2}Ca$ is independent of $j$. We have

\begin{thm}\label{assca}
The Borel Adams spectral sequence for $\pi_{*,j}^{C_2}Ca$
\[
\bigoplus_{n=0}^{\infty}\pi_{i,j}^{C_2}\left(Ca\wedge H_{C_2}\wedge\overline{H_{C_2}}^{\wedge n}\right)^h \Rightarrow \pi_{i,j}^{C_2}Ca
\]
is isomorphic to the classical Adams spectral sequence for classical sphere
\[
\bigoplus_{n=0}^{\infty}\pi_{i-j-1}\left(S^{-j-1}\wedge H_{\textup{cl}}\wedge\overline{H_{\textup{cl}}}^{\wedge n}\right) \Rightarrow \pi_{i-j-1}S^{-j-1}
\]
from their $E_1$-pages.
\end{thm}

\begin{proof}
We have the commutative diagram
\[
\xymatrix{
\pi_{i+1,j+1}^{C_2}\Map\left((EC_2)_+,\overline{H_{C_2}}^{\wedge n}\right) \ar[r]^(0.53){\cong}\ar[d]_a & \pi_{i-j-1}\varprojlim\limits_k\left(\PP_{-k}^{-j-2}\wedge\overline{H_{\textup{cl}}}^{\wedge n}\right) \ar[d] \\
\pi_{i,j}^{C_2}\Map\left((EC_2)_+,\overline{H_{C_2}}^{\wedge n}\right) \ar[r]^{\cong}\ar[d] & \pi_{i-j-1}\varprojlim\limits_k\left(\PP_{-k}^{-j-1}\wedge\overline{H_{\textup{cl}}}^{\wedge n}\right) \ar[d] \\
\pi_{i,j}^{C_2}\Map\left((EC_2)_+,Ca\wedge\overline{H_{C_2}}^{\wedge n}\right) \ar@{-->}[r]^(0.55){\cong} & \pi_{i-j-1}\left(S^{-j-1}\wedge\overline{H_{\textup{cl}}}^{\wedge n}\right), \\
}
\]
and a similar one for $H_{C_2}\wedge\overline{H_{C_2}}^{\wedge n}$, which are compatible in the sense that there is a commutative diagram
\[
\xymatrix{
\pi_{i,j}^{C_2}\Map\left((EC_2)_+,Ca\wedge\overline{H_{C_2}}^{\wedge n}\right) \ar[r]^(0.55){\cong} \ar[d] & \pi_{i-j-1}\left(S^{-j-1}\wedge\overline{H_{\textup{cl}}}^{\wedge n}\right) \ar[d] \\
\pi_{i,j}^{C_2}\Map\left((EC_2)_+,Ca\wedge H_{C_2}\wedge\overline{H_{C_2}}^{\wedge n}\right) \ar[r]^(0.55){\cong} & \pi_{i-j-1}\left(S^{-j-1}\wedge H_{\textup{cl}}\wedge\overline{H_{\textup{cl}}}^{\wedge n}\right). \\
}
\]
Therefore, the two spectral sequences are isomorphic.
\end{proof}

Finally, for the upper right sides of the squares, we have

\begin{thm}\label{rhobss}
The algebraic $a$-Bockstein spectral sequence for the $E_2$-term of the Borel Adams spectral sequence for the sphere
\begin{equation}\label{abss}
\bigoplus_{n=0}^{\infty}\Ext_{\A_{*,*}^h}^{s,t+n,w+n}\left(\left(H_{C_2}\right)^h_{*,*}/a\right)\Rightarrow \Ext_{\A_{*,*}^h}^{s,t,w}\left(\left(H_{C_2}\right)^h_{*,*}\right)
\end{equation}
is isomorphic to the algebraic Atiyah-Hirzebruch spectral sequence for the classical $\Ext$ groups
\begin{equation}\label{aahss}
\bigoplus_{n=0}^{\infty}\Ext_{\A_*}^{s,t-w-1}\left(H_*S^{-w-1-n}\right) \Rightarrow \Ext_{\A_*}^{s,t-w-1}\left(\varprojlim\limits_kH_*\PP_{-k}^{-w-1}\right)
\end{equation}
from their $E_1$-pages.
\end{thm}

\begin{proof}
When we consider the $E_1$-pages of the Borel Adams spectral sequence and the limit of the classical Adams spectral sequences, we will have the commutative square
\[
\xymatrix@C=0.4cm{
\pi_{t-s+1,w+1}^{C_2}\Map\left((EC_2)_+,H_{C_2}\wedge\overline{H_{C_2}}^{\wedge s}\right) \ar[r]^(0.52){\cong}\ar[d]_a & \pi_{t-s-w-1}\varprojlim\limits_k\left(\PP_{-k}^{-w-2}\wedge H_{\textup{cl}}\wedge\overline{H_{\textup{cl}}}^{\wedge s}\right) \ar[d] \\
\pi_{t-s,w}^{C_2}\Map\left((EC_2)_+,H_{C_2}\wedge\overline{H_{C_2}}^{\wedge s}\right) \ar[r]^(0.49){\cong} & \pi_{t-s-w-1}\varprojlim\limits_k\left(\PP_{-k}^{-w-1}\wedge H_{\textup{cl}}\wedge\overline{H_{\textup{cl}}}^{\wedge s}\right). \\
}
\]
Both $d_1$'s are induced by the nonequivariant map $H_{\textup{cl}}\wedge\overline{H_{\textup{cl}}}^{\wedge s}\to\Sigma H_{\textup{cl}}\wedge\overline{H_{\textup{cl}}}^{\wedge (s+1)}$, hence the square is compatible with the differentials. Note that $H_{\textup{cl}}\wedge\overline{H_{\textup{cl}}}^{\wedge s}$ splits as the direct sum of suspensions of $H_{\textup{cl}}$, so we have
\[\pi_{t-s-w-1}\left(\PP_{-k}^{-w-1}\wedge H_{\textup{cl}}\wedge\overline{H_{\textup{cl}}}^{\wedge s}\right) \cong C^{s,t-w-1}_{\A_*}\left(H_*\PP_{-k}^{-w-1}\right).\]
Since the map
\[H_*\PP_{-k-1}^{-w-1}\to H_*\PP_{-k}^{-w-1}\]
is surjective, the Mittag-Leffler condition is satisfied, leading to the vanishing of the $\varprojlim^1$ term, so we have
\[\pi_{t-s-w-1}\varprojlim\limits_k\left(\PP_{-k}^{-w-1}\wedge H_{\textup{cl}}\wedge\overline{H_{\textup{cl}}}^{\wedge s}\right) \cong C^{s,t-w-1}_{\A_*}\left(\varprojlim\limits_kH_*\PP_{-k}^{-w-1}\right).\]
Passing to homology, we obtain the isomorphisms on the $E_2$-terms which fit into the commutative square
\[
\xymatrix{
\Ext_{\A_{*,*}^h}^{s,t+1,w+1}\left(\left(H_{C_2}\right)^h_{*,*}\right) \ar[r]^(0.46){\cong}\ar[d]_a & \Ext_{\A_*}^{s,t-w-1}\left(\varprojlim\limits_kH_*\PP_{-k}^{-w-2}\right) \ar[d] \\
\Ext_{\A_{*,*}^h}^{s,t,w}\left(\left(H_{C_2}\right)^h_{*,*}\right) \ar[r]^(0.42){\cong} & \Ext_{\A_*}^{s,t-w-1}\left(\varprojlim\limits_kH_*\PP_{-k}^{-w-1}\right). \\
}
\]
Therefore, the two spectral sequences are isomorphic.
\end{proof}

\section{On the limit of the Adams spectral sequences}\label{sec-limit}
In this section, we will give another spectral sequence converging to $\pi_*\PP_{-\infty}^*$, which is isomorphic to the limit of the Adams spectral sequences (as in Theorem \ref{borandlim}), but easier to compute.

First we will reveal some properties of $\PP_{-\infty}^{-w-1}$. By taking the inverse limit of the fiber sequences
\[\PP_{-k}^{-w-1}\to\PP_{-k}^{\infty}\to\PP_{-w}^{\infty}\]
with respect to $k$, we have the fiber sequence
\[\PP_{-\infty}^{-w-1}\to\PP_{-\infty}^{\infty}\to\PP_{-w}^{\infty},\]
where Lin (\cite{lin}) proved that there is a ($2$-adic) homotopy equivalence
\[S^{-1}\xrightarrow{\simeq}\PP_{-\infty}^{\infty}.\]
Therefore, $\PP_{-\infty}^{-w-1}$ is a bounded below spectrum of finite type.

When $w\le0$, $\PP_{-w}^{\infty}$ is $0$-connective, hence $\pi_{-1}\PP_{-w}^{\infty}=0$. Therefore, the map $\PP_{-\infty}^{\infty}\to\PP_{-w}^{\infty}$ is null, and
\[\PP_{-\infty}^{-w-1}\simeq S^{-1}\vee\Sigma^{-1}\PP_{-w}^{\infty}.\]

When $w>0$, the fiber sequence
\[\PP_{-\infty}^{-2}\to\PP_{-\infty}^{\infty}\to\PP_{-1}^{\infty}\]
implies that $\PP_{-\infty}^{-2}$ is $(-2)$-connective. By Lemma \ref{pands}, we have
\[\pi_{-2}\PP_{-\infty}^{-2}\cong\pi_{0,1}^{C_2}\SC=0,\]
where the second equality is proved in \cite[Thm.~7.6(ii)]{ai}. Thus $\PP_{-\infty}^{-2}$ is $(-1)$-connective. Then the long exact sequence on homology
\[\cdots\to H_{-1}\PP_{-\infty}^{\infty}\to H_{-1}\PP_{-1}^{\infty}\to H_{-2}\PP_{-\infty}^{-2}\to\cdots\]
implies that $\PP_{-\infty}^{\infty}\to\PP_{-1}^{\infty}$ induces an isomorphism on $H_{-1}$. Applying $H_{-1}$ to the commutative diagram
\[
\xymatrix{
\PP_{-\infty}^{\infty} \ar[r]\ar[rd] & \PP_{-w}^{\infty} \ar[d] \\
 & \PP_{-1}^{\infty}, \\
}
\]
we know that $\PP_{-\infty}^{\infty}\to\PP_{-w}^{\infty}$ induces an isomorphism on $H_{-1}$. Since $H_*\PP_{-\infty}^\infty$ is concentrated in degree $-1$, we have a short exact sequence of $\A_*$-comodules
\[0\to H_*\PP_{-\infty}^{\infty}\to H_*\PP_{-w}^{\infty}\to H_*\Sigma\PP_{-\infty}^{-w-1}\to0.\]

When $w>0$, the following theorem relates the limit of the Adams spectral sequences for $\PP_{-\infty}^{-w-1}$ appeared in Theorem \ref{borandlim} to the Adams spectral sequence for $\Sigma\PP_{-\infty}^{-w-1}$. Note that they have different constructions, where in the latter ones we simply apply the Adams resolution to the (bounded below) spectrum $\Sigma\PP_{-\infty}^{-w-1}$.

\begin{thm}\label{posw}
For $w>0$, the limit of the classical Adams spectral sequences
\[\varprojlim\limits_k E_r^{s,t-w-1}\left(\PP_{-k}^{-w-1}\right)\Rightarrow \pi_{t-s-w-1}\PP_{-\infty}^{-w-1}\]
is isomorphic to the classical Adams spectral sequence
\begin{equation}\label{ass>0}
E_r^{s-1,t-w-1}\left(\Sigma\PP_{-\infty}^{-w-1}\right)\Rightarrow \pi_{t-s-w}\Sigma\PP_{-\infty}^{-w-1}
\end{equation}
from their $E_2$-pages.
\end{thm}

\begin{proof}
For $k>w+1$, the fiber sequence
\[\PP_{-k}^{-w-1}\to\Sigma\PP_{-\infty}^{-k-1}\to\Sigma\PP_{-\infty}^{-w-1}\]
induces a short exact sequence on $H_*$. By the geometric boundary theorem (e.g. \cite{green}), there is a map of the Adams spectral sequences
\[E_r^{s-1,t-w-1}\left(\Sigma\PP_{-\infty}^{-w-1}\right)\to E_r^{s,t-w-1}\left(\PP_{-k}^{-w-1}\right).\]
By taking limits, we get the map of spectral sequences
\[E_r^{s-1,t-w-1}\left(\Sigma\PP_{-\infty}^{-w-1}\right)\to \varprojlim\limits_k E_r^{s,t-w-1}\left(\PP_{-k}^{-w-1}\right).\]
On the $E_2$-pages, the map is induced by the connecting homomorphism on $\Ext$ groups
\[\Ext_{\A_*}^{s-1,t-w-1}\left(H_*\Sigma\PP_{-\infty}^{-w-1}\right)\to \varprojlim\limits_k \Ext_{\A_*}^{s,t-w-1}\left(H_*\PP_{-k}^{-w-1}\right)\]
from the short exact sequence of $\A_*$-comodules
\[0\to\varprojlim\limits_kH_*\PP_{-k}^{-w-1}\to\varprojlim\limits_kH_*\Sigma\PP_{-\infty}^{-k-1}\to H_*\Sigma\PP_{-\infty}^{-w-1}\to0,\]
where the Mittag-Leffler condition is satisfied since the map
\[H_*\PP_{-k-1}^{-w-1}\to H_*\PP_{-k}^{-w-1}\]
is surjective. Now it suffices to show that the map (\ref{extlim}) is an isomorphism. Consider the short exact sequence of $\A_*$-comodules
\[0\to H_*\PP_{-\infty}^{\infty}\to\varprojlim\limits_kH_*\PP_{-k}^{\infty}\to\varprojlim\limits_kH_*\Sigma\PP_{-\infty}^{-k-1}\to0,\]
where the injective map induces an isomorphism on $\Ext_{\A_*}^{*,*}$, by \cite{ldma}. Therefore, 
\[\Ext_{\A_*}^{*,*}\left(\varprojlim\limits_kH_*\Sigma\PP_{-\infty}^{-k-1}\right)=0,\]
and hence, the map on the $E_2$-pages is indeed an isomorphism, which completes the proof.
\end{proof}

When $w\le0$, the following theorem suggests that the limit of the Adams spectral sequences splits as the direct sum of two classical Adams spectral sequences.

\begin{thm}\label{negw}
For $w\le0$, the limit of the classical Adams spectral sequences
\[\varprojlim\limits_k E_r^{s,t-w-1}\left(\PP_{-k}^{-w-1}\right)\Rightarrow \pi_{t-s-w-1}\PP_{-\infty}^{-w-1}\]
and the direct sum of two classical Adams spectral sequences
\begin{equation}\label{ass<=0}
E_r^{s-1,t-w-1}\left(\PP_{-w}^{\infty}\right)\oplus E_r^{s,t-w-1}\left(S^{-1}\right)\Rightarrow \pi_{t-s-w}\PP_{-w}^{\infty}\oplus\pi_{t-s-w-1}S^{-1}
\end{equation}
are isomorphic from $E_2$-pages.
\end{thm}

\begin{proof}
For $k>w+1$, the fiber sequence
\[\PP_{-k}^{-w-1}\to\PP_{-k}^{\infty}\to\PP_{-w}^{\infty}\]
induces a short exact sequence on $H_*$, and hence a long exact sequence of $\Ext$ groups. Note that we have a commutative diagram
\[
\xymatrix{
 & \varprojlim\limits_k\Ext_{\A_*}^{s,t-w-1}\left(H_*\PP_{-k}^{\infty}\right) \ar[rd] & \\
\Ext_{\A_*}^{s,t-w-1}\left(H_*\PP_{-\infty}^{\infty}\right) \ar[ru]^(0.4){\cong} \ar[rr]^0 & & \Ext_{\A_*}^{s,t-w-1}\left(H_*\PP_{-w}^{\infty}\right), \\
}
\]
where the isomorphism is proved in \cite{ldma}, and the horizontal map is trivial since it is induced from the null map
\[S^{-1}\simeq\PP_{-\infty}^{\infty}\to\PP_{-w}^{\infty}.\]
Therefore, the long exact sequence of $\Ext$ groups splits into short exact sequences
\[\begin{split}
0\to\Ext_{\A_*}^{s-1,t-w-1}\left(H_*\PP_{-w}^{\infty}\right)&\to\varprojlim\limits_k\Ext_{\A_*}^{s,t-w-1}\left(H_*\PP_{-k}^{-w-1}\right)\\
&\to\varprojlim\limits_k\Ext_{\A_*}^{s,t-w-1}\left(H_*\PP_{-k}^{\infty}\right)\to0.
\end{split}\]
Note that the maps of $\Ext$ groups are induced from maps in the fiber sequence, hence can be extended to maps of spectral sequences via the geometric boundary theorem. 

Fix a splitting
\[\PP_{-\infty}^{\infty}\to\PP_{-\infty}^{-w-1}\simeq\PP_{-\infty}^{\infty}\vee\Sigma^{-1}\PP_{-w}^{\infty}.\]
Then the composite map
\[\PP_{-\infty}^{\infty}\to\PP_{-\infty}^{-w-1}\to\PP_{-k}^{-w-1}\]
fits into the commutative diagram
\[
\xymatrix{
\PP_{-\infty}^{\infty} \ar[r]\ar[rd]\ar@/^1.5pc/[rr]^{\textup{id}} & \PP_{-\infty}^{-w-1} \ar[r]\ar[d] & \PP_{-\infty}^{\infty} \ar[d] \\
 & \PP_{-k}^{-w-1} \ar[r] & \PP_{-k}^{\infty}, \\
}
\]
and hence gives us the lifting maps that are compatible with limits with respect to $k$. Therefore, it induces a map of spectral sequences that splits the short exact sequence of the $\Ext$ groups on the $E_2$-pages:
\[
\xymatrix{
 & \Ext_{\A_*}^{s,t-w-1}\left(H_*\PP_{-\infty}^{\infty}\right) \ar@{-->}[ld] \ar[d]^{\cong} &  \\
\varprojlim\limits_k\Ext_{\A_*}^{s,t-w-1}\left(H_*\PP_{-k}^{-w-1}\right) \ar[r] & \varprojlim\limits_k\Ext_{\A_*}^{s,t-w-1}\left(H_*\PP_{-k}^{\infty}\right) \ar[r] & 0. \\
}
\]
Thus, the limit of the Adams spectral sequences splits as a direct sum of two Adams spectral sequences.
\end{proof}

These two theorems simplify the computation of the limit of the Adams spectral sequences mentioned before, and hence the computation of the Borel Adams spectral sequence. In particular, the computation of the $E_2$-term can be carried out by a computer program through the Curtis algorithm. See Section \ref{sec-names} for more information about computations.

\section{A $4\times4$ lemma}\label{sec-lemma}
In this section, we will prove a $4\times4$ lemma, which is a generalization of \cite[Lem.~9.3.2]{am17}. The results of this section will be used in Section \ref{sec-genuine}.

Let $\mathcal C$ be a stable $\infty$-category. Let $\mathcal E\subset\textup{Fun}\left(\Delta^1\times\Delta^2,\mathcal C\right)$ be the full subcategory spanned by the diagrams
\[
\xymatrix{
X \ar[r]^{f} \ar[d] & Y \ar[r] \ar[d]^{g} & 0 \ar[d] \\
0' \ar[r] & Z \ar[r]^{h} & W,
}
\]
where $0,0'\in\mathcal C$ are zero objects and both squares are pushout diagrams in $\mathcal C$. Let $e:\mathcal E \to \FC$ be the restriction to the upper left horizontal arrow, which is shown in the proof of \cite[Thm.~1.1.2.14]{ha} to be a trivial fibration. By \cite[Thm.~1.1.2.14]{ha}, the homotopy category $h\mathcal C$ is a triangulated category, where the diagram
\[X\xrightarrow{[f]}Y\xrightarrow{[g]}Z\xrightarrow{[h]}\Sigma X\]
in $h\mathcal C$ is a distinguished triangle if there exists a diagram
\[
\xymatrix{
X \ar[r]^{f} \ar[d] & Y \ar[r] \ar[d]^{g} & 0 \ar[d] \\
0' \ar[r] & Z \ar[r]^{h} & W,
}
\]
in $\mathcal C$, where $[f]$ and $[g]$ are represented by $f$ and $g$ respectively, and $[h]$ is the composition of the homotopy class of $h$ with the equivalence $W\simeq\Sigma X$ determined by the outer rectangle.

\begin{lem}\label{lem6}
Let
\[
\sigma_1:\begin{gathered}\xymatrix{
X_1^1 \ar[r]^{f_1} \ar[d] & X_1^2 \ar[r] \ar[d]^{g_1} & 0 \ar[d] \\
0 \ar[r] & X_1^3 \ar[r]^{h_1} & X_1^4,\\}\end{gathered}\quad \sigma_2:\begin{gathered}\xymatrix{
X_2^1 \ar[r]^{f_2} \ar[d] & X_2^2 \ar[r] \ar[d]^{g_2} & 0 \ar[d] \\
0 \ar[r] & X_2^3 \ar[r]^{h_2} & X_2^4, \\}\end{gathered}\]
\[\sigma^1:\begin{gathered}\xymatrix{
X_1^1 \ar[r]^{f^1} \ar[d] & X_2^1 \ar[r] \ar[d]^{g^1} & 0 \ar[d] \\
0 \ar[r] & X_3^1 \ar[r]^{h^1} & X_4^1, \\}\end{gathered}\quad \sigma^2:\begin{gathered}\xymatrix{
X_1^2 \ar[r]^{f^2} \ar[d] & X_2^2 \ar[r] \ar[d]^{g^2} & 0 \ar[d] \\
0 \ar[r] & X_3^2 \ar[r]^{h^2} & X_4^2 \\}\end{gathered}
\]
be objects in $\mathcal E$. Let $\phi:\sigma_1\to\sigma_2$ and $\psi:\sigma^1\to\sigma^2$ be morphisms in $\mathcal E$ such that the transpose of $e(\psi)$ is
\[
e(\phi):\begin{gathered}\xymatrix{
X_1^1 \ar[r]^{f_1} \ar[d]^{f^1} & X_1^2 \ar[d]^{f^2} \\
X_2^1 \ar[r]^{f_2} & X_2^2 \\
}\end{gathered}
\]
regarded as a diagram in $\mathcal C$. Consider the induced diagram in $h\mathcal C$:
\[
\xymatrix{
X_1^1 \ar[r]^{[f_1]} \ar[d]^{[f^1]} & X_1^2 \ar[r]^{[g_1]} \ar[d]^{[f^2]} & X_1^3 \ar[r]^{[h_1]} \ar[d]^{[\phi^3]} & \Sigma X_1^1 \ar[d]^{\Sigma [f^1]} \\
X_2^1 \ar[r]^{[f_2]} \ar[d]^{[g^1]} & X_2^2 \ar[r]^{[g_2]} \ar[d]^{[g^2]} & X_2^3 \ar[r]^{[h_2]} & \Sigma X_2^1 \\
X_3^1 \ar[r]^{[\psi_3]} \ar[d]^{[h^1]} & X_3^2 \ar[d]^{[h^2]} & & \\
\Sigma X_1^1 \ar[r]^{\Sigma [f_1]} & \Sigma X_1^2. & & \\
}
\]
Fix $C\in\mathcal C$ and $\alpha_3^1\in[C,X_3^1]$ such that $[h^2][\psi_3]\alpha_3^1=0$. Then there exist liftings $\alpha_2^2\in[C,X_2^2]$ of $[\psi_3]\alpha_3^1$, and $\alpha_1^3\in[C,X_1^3]$ of $[g_2]\alpha_2^2$, such that
\[[h_1]\alpha_1^3=-[h^1]\alpha_3^1\in[C,\Sigma X_1^1].\]
\end{lem}

\begin{proof}
Starting from the diagram
\[
\xymatrix{
X_1^1 \ar[r]^{f_1} \ar[d]^{f^1} & X_1^2 \\
X_2^1, & \\
}
\]
we can apply \cite[Prop.~4.3.2.15]{htt} four times to construct the following diagram in $\mathcal C$:
\begin{equation}\label{kan4}
\begin{gathered}
\xymatrix{
X_1^1 \ar[r]^{f_1} \ar[d]^{f^1} & X_1^2 \ar[r] \ar[d]^{p_1^2} & 0 \ar[d] & \\
X_2^1 \ar[r]^{p_2^1} \ar[d] & P \ar[r]^{q_3^1} \ar[d]^{q_1^3} & Y_3^1 \ar[r] \ar[d]^{r_3^1} & 0 \ar[d] \\
0 \ar[r] & Y_1^3 \ar[r]^{r_1^3} \ar[d] & Z_1^1 \ar[r] \ar[d] & Z_2^1 \\
 & 0 \ar[r] & Z_1^2, & \\
}
\end{gathered}
\end{equation}
where each square in the diagram is a pushout. Since $e:\mathcal E\to\FC$ is a trivial fibration, we can extend the constant $3$-simplex on $f_1:X_1^1\to X_1^2$ in $\FC$ to a $3$-simplex in $\mathcal E$:
\[\def\b#1{\save[].[drr]!C="b#1"*+[F--]\frm{}\restore}
\xymatrix{
\b1 X_1^1 \ar[r]^{f_1} \ar[d] & X_1^2 \ar[r] \ar[d] & 0 \ar[d] & \b2 X_1^1 \ar[r]^{f_1} \ar[d] & X_1^2 \ar[r] \ar[d]^{g_1} & 0 \ar[d] \\
0 \ar[r] & Y_1^3 \ar[r]^{r_1^3} & Z_1^1 & 0 \ar[r] & X_1^3 \ar[r]^{h_1} & X_1^4 \\
\b3 X_1^1 \ar[r]^{f_1} \ar[d] & X_1^2 \ar[r] \ar[d] & 0 \ar[d] & \b4 X_1^1 \ar[r]^{f_1} \ar[d] & X_1^2 \ar[r] \ar[d]^{g_1} & 0 \ar[d] \\
0 \ar[r] & Y_1^3 \ar[r]^{r_1^3} & Z_1^1 & 0 \ar[r] & X_1^3 \ar[r]^{h_1} & X_1^4.
\ar "b1";"b2"^{i_1}
\ar "b2";"b3"^{j_1}
\ar "b3";"b4"^{i_1}
\ar "b1";"b3"_{id}
\ar "b2";"b4"^{id}
}
\]
Then we have morphisms $[i_1^3]:Y_1^3\to X_1^3$ and $[j_1^3]:X_1^3\to Y_1^3$ in $h\mathcal C$ that are inverse to each other. Similarly, we can extend the constant $3$-simplex on $f^1:X_1^1\to X_2^1$ in $\FC$ to a $3$-simplex in $\mathcal E$:
\[\def\b#1{\save[].[drr]!C="b#1"*+[F--]\frm{}\restore}
\xymatrix{
\b1 X_1^1 \ar[r]^{f^1} \ar[d] & X_2^1 \ar[r] \ar[d] & 0 \ar[d] & \b2 X_1^1 \ar[r]^{f^1} \ar[d] & X_2^1 \ar[r] \ar[d]^{g^1} & 0 \ar[d] \\
0 \ar[r] & Y_3^1 \ar[r]^{r_3^1} & Z_1^1 & 0 \ar[r] & X_3^1 \ar[r]^{h^1} & X_4^1 \\
\b3 X_1^1 \ar[r]^{f^1} \ar[d] & X_2^1 \ar[r] \ar[d] & 0 \ar[d] & \b4 X_1^1 \ar[r]^{f^1} \ar[d] & X_2^1 \ar[r] \ar[d]^{g^1} & 0 \ar[d] \\
0 \ar[r] & Y_3^1 \ar[r]^{r_3^1} & Z_1^1 & 0 \ar[r] & X_3^1 \ar[r]^{h^1} & X_4^1. 
\ar "b1";"b2"^{i^1}
\ar "b2";"b3"^{j^1}
\ar "b3";"b4"^{i^1}
\ar "b1";"b3"_{id}
\ar "b2";"b4"^{id}
}
\]
Consequently, we have morphisms $[i_3^1]:Y_3^1\to X_3^1$ and $[j_3^1]:X_3^1\to Y_3^1$ in $h\mathcal C$ that are inverse to each other.

Let $[l_1^3]=[i_1^3]\circ[q_1^3]:P\to X_1^3$, and $[l_3^1]=[i_3^1]\circ[q_3^1]:P\to X_3^1$ be morphisms in $h\mathcal C$. Now we claim that there is a commutative square in $h\mathcal C$:
\begin{equation}\label{diag1331}
\begin{gathered}
\xymatrix{
P \ar[r]^{[l_1^3]} \ar[d]^{[l_3^1]} & X_1^3 \ar[d]^{[h_1]} \\
X_3^1 \ar[r]^{-[h^1]} & \Sigma X_1^1.
}
\end{gathered}
\end{equation}
Note that when we let the square 
\begin{equation}\label{susp}
\begin{gathered}
\xymatrix{
X_1^1 \ar[r] \ar[d] & 0 \ar[d] \\
0 \ar[r] & Z_1^1
}
\end{gathered}
\end{equation}
appeared in (\ref{kan4}) exhibit $Z_1^1\in\mathcal C$ as the representative of $\Sigma X_1^1$, the map $[h_1]$ in (\ref{diag1331}) will be represented by the composite of $[i_1^4]^{-1}:X_1^4\to Z_1^1$ and $[h_1]:X_1^3\to X_1^4$, and the map $[h^1]$ in (\ref{diag1331}) will be represented by the composite of $-[i_4^1]^{-1}:X_4^1\to Z_1^1$ and $[h^1]:X_3^1\to X_4^1$. The minus sign appears in the latter case since the square (\ref{susp}) and the outer square of 
\[
\xymatrix{
X_1^1 \ar[r]^{f^1} \ar[d] & X_2^1 \ar[r] \ar[d] & 0 \ar[d]\\
0 \ar[r] & Y_3^1 \ar[r]^{r_3^1} & Z_1^1
}
\]
classify two morphisms in the Abelian group $\textup{Hom}_{h\mathcal C}\left(\Sigma X_1^1,Z_1^1\right)$ which are inverses of each other by \cite[Lem.~1.1.2.10]{ha}.

Since there is a diagram in $\mathcal C$:
\[
\xymatrix{
Y_1^3 \ar[r]^{i_1^3} \ar[d]^{r_1^3} & X_1^3 \ar[d]^{h_1} \\
Z_1^1 \ar[r]^{i_1^4} & X_1^4,
}
\]
we have
\[[i_1^4]^{-1}\circ[h_1]\circ[l_1^3]=[i_1^4]^{-1}\circ[h_1]\circ[i_1^3]\circ[q_1^3]=[r_1^3]\circ[q_1^3].\]
Similarly, we have
\[[i_4^1]^{-1}\circ[h^1]\circ[l_3^1]=[i_4^1]^{-1}\circ[h^1]\circ[i_3^1]\circ[q_3^1]=[r_3^1]\circ[q_3^1].\]
Furthermore, $[r_1^3]\circ[q_1^3]=[r_3^1]\circ[q_3^1]$ since there is a diagram in $\mathcal C$ (as appeared in (\ref{kan4})):
\[
\xymatrix{
P \ar[r]^{q_3^1} \ar[d]^{q_1^3} & Y_3^1 \ar[d]^{r_3^1} \\
Y_1^3 \ar[r]^{r_1^3} & Z_1^1,
}
\]
and the claim follows.

Since the square
\begin{equation}\label{po}
\begin{gathered}
\xymatrix{
X_1^1 \ar[r]^{f_1} \ar[d]^{f^1} & X_1^2 \ar[d]^{p_1^2} \\
X_2^1 \ar[r]^{p_2^1} & P
}
\end{gathered}
\end{equation}
is a pushout, there is a $2$-simplex in $\FC$:
\[
\xymatrix{
 & X_1^2 \ar[d]_(0.6){p_1^2} \ar[ddr]^{id} & \\
 & P \ar[ddr]|!{[dl];[dr]}\hole^(0.4){l_2^2} & \\
X_1^1 \ar[uur]^{f_1} \ar[rr]_{f_1} \ar[d]_{f^1} & & X_1^2 \ar[d]^{f^2} \\
X_2^1 \ar[uur]|!{[u];[urr]}\hole^(0.6){p_2^1} \ar[rr]_{f_2} & & X_2^2. 
}
\]
The diagram (\ref{kan4}) induces a morphism in $\mathcal E$
\[k_1:\begin{gathered}\xymatrix{
X_1^1 \ar[r]^{f^1} \ar[d] & X_2^1 \ar[r] \ar[d] & 0 \ar[d] \\
0 \ar[r] & Y_3^1 \ar[r]^{r_3^1} & Z_1^1\\}\end{gathered}\to\begin{gathered}\xymatrix{
X_1^2 \ar[r]^{p_1^2} \ar[d] & P \ar[r] \ar[d]^{q_3^1} & 0 \ar[d] \\
0 \ar[r] & Y_3^1 \ar[r] & Z_1^2\\}\end{gathered}\]
extending the square (\ref{po}), which is the identity restricted to $Y_3^1$. Since $e:\mathcal E\to \FC$ is a trivial fibration, we can complete the following diagram in $\mathcal E$ to a $3$-simplex:
 \[\def\b#1{\save[].[drr]!C="b#1"*+[F--]\frm{}\restore}
\xymatrix{
\b2 X_1^1 \ar[r]^{f^1} \ar[d] & X_2^1 \ar[r] \ar[d] & 0 \ar[d] & \b1 X_1^1 \ar[r]^{f^1} \ar[d] & X_2^1 \ar[r] \ar[d]^{g^1} & 0 \ar[d] \\
0 \ar[r] & Y_3^1 \ar[r]^{r_3^1} & Z_1^1 & 0 \ar[r] & X_3^1 \ar[r]^{h^1} & X_4^1 \\
\b3 X_1^2 \ar[r]^{p_1^2} \ar[d] & P \ar[r] \ar[d]^{q_3^1} & 0 \ar[d] & \b4 X_1^2 \ar[r]^{f^2} \ar[d] & X_2^2 \ar[r] \ar[d]^{g^2} & 0 \ar[d] \\
0 \ar[r] & Y_3^1 \ar[r] & Z_1^2 & 0 \ar[r] & X_3^2 \ar[r]^{h^2} & X_4^2. 
\ar "b1";"b2"_{j^1}
\ar "b2";"b3"_{k_1}
\ar "b1";"b4"^{\psi}
\ar @{-->} "b3";"b4"_{\varphi^2}
}
\]
Therefore, there exists a morphism $\varphi_3^2:Y_3^1\to X_3^2$ such that $[\psi_3]=[\varphi_3^2]\circ[j_3^1]$ and that there is a diagram:
\[
\xymatrix{
P \ar[r]^{l_2^2} \ar[d]^{q_3^1} & X_2^2 \ar[d]^{g^2} \\
Y_3^1 \ar[r]^{\varphi_3^2} & X_3^2.
}
\]
Then we have
\[[\psi_3]\circ[l_3^1]=[\varphi_3^2]\circ[j_3^1]\circ[i_3^1]\circ[q_3^1]=[\varphi_3^2]\circ[q_3^1]=[g^2]\circ[l_2^2].\]
Consequently, we have the commutative square in $h\mathcal C$:
\begin{equation}\label{diag2231}
\begin{gathered}
\xymatrix{
P \ar[r]^{[l_3^1]} \ar[d]^{[l_2^2]} & X_3^1 \ar[d]^{[\psi_3]} \\
X_2^2 \ar[r]^{[g^2]} & X_3^2.
}
\end{gathered}
\end{equation}
Similarly, we have the the commutative square in $h\mathcal C$:
\begin{equation}\label{diag1322}
\begin{gathered}
\xymatrix{
P \ar[r]^{[l_1^3]} \ar[d]^{[l_2^2]} & X_1^3 \ar[d]^{[\phi^3]} \\
X_2^2 \ar[r]^{[g_2]} & X_2^3.
}
\end{gathered}
\end{equation}

Since
\[\varphi^2:\begin{gathered}\xymatrix{X_1^2 \ar[r]^{p_1^2} \ar[d] & P \ar[r] \ar[d]^{q_3^1} & 0 \ar[d] \\ 0 \ar[r] & Y_3^1 \ar[r] & Z_1^2}\end{gathered}\to\begin{gathered}\xymatrix{X_1^2 \ar[r]^{f^2} \ar[d] & X_2^2 \ar[r] \ar[d]^{g^2} & 0 \ar[d] \\ 0 \ar[r] & X_3^2 \ar[r]^{h^2} & X_4^2}\end{gathered}\]
induces the identity morphism on $X_1^2$, the morphism $\varphi_4^2:Z_1^2\to X_4^2$ is a homotopy equivalence. Since 
\[0=[h^2][\psi_3]\alpha_3^1=[h^2][\varphi_3^2][j_3^1]\alpha_3^1,\]
$[j_3^1]\alpha_3^1$ can be lifted along $[q_3^1]$ to $\alpha\in[C,P]$, and we have
\[\alpha_3^1=[i_3^1][j_3^1]\alpha_3^1=[i_3^1][q_3^1]\alpha=[l_3^1]\alpha.\]
Let $\alpha_2^2=[l_2^2]\alpha$, and $\alpha_1^3=[l_1^3]\alpha$. We have $[g^2]\alpha_2^2=[\psi_3]\alpha_3^1$ by (\ref{diag2231}), $[\phi^3]\alpha_1^3=[g_2]\alpha_2^2$ by (\ref{diag1322}), and $[h_1]\alpha_1^3=-[h^1]\alpha_3^1$ by (\ref{diag1331}).
\end{proof}

\begin{rmk}
Lemma \ref{lem6} can be regarded as a corrected version of \cite[Lem.~4.5]{hhr17}. The counterexample of \cite[Lem.~4.5]{hhr17} is given in the footnote in \cite[p.~29]{msz}. The reason behind the existence of the counterexample is that \cite[Lem.~4.5]{hhr17} does not make use of the homotopy coherence data, which is also the reason that Lemma \ref{lem6} in our article is stated in the language of stable $\infty$-category.
\end{rmk}

\begin{cor}\label{brprop}
With the same condition as in Lemma \ref{lem6}, let $\beta_3^2\in[C,X_3^2]$ such that it can be lifted along both $[\psi_3]$ and $[g^2]$. Then
\[[h_1]\left([\phi^3]\right)^{-1}[g_2]\left([g^2]\right)^{-1}(\beta_3^2)=-[h^1]\left([\psi_3]\right)^{-1}(\beta_3^2)\]
as subsets of $[C,\Sigma X_1^1]$.
\end{cor}

\begin{proof}
By Lemma \ref{lem6}, we have
\[-[h^1]\left([\psi_3]\right)^{-1}(\beta_3^2)\subset[h_1]\left([\phi^3]\right)^{-1}[g_2]\left([g^2]\right)^{-1}(\beta_3^2).\]
On the other hand, suppose that we have liftings $\beta_2^2\in[C,X_2^2]$ of $\beta_3^2$, and $\beta_1^3\in[C,X_1^3]$ of $[g_2]\beta_2^2$, it suffices to show that there exists a lifting $\beta_3^1\in[C,X_3^1]$ of $\beta_3^2$ such that $[h^1]\beta_3^1=-[h_1]\beta_1^3$.

Consider the map of fiber sequences:
\[
\xymatrix@C=1cm{
X_2^1 \ar[r]^{[p_2^1]} \ar[d]_{id} & P \ar[r]^{[l_1^3]} \ar[d]_{[l_2^2]} & X_1^3 \ar[r]^{[h_2]\circ[\phi^3]} \ar[d]^{[\phi^3]} & \Sigma X_2^1 \ar[d]^{id} \\
X_2^1 \ar[r]_{[f_2]} & X_2^2 \ar[r]_{[g_2]} & X_2^3 \ar[r]_{[h_2]} & \Sigma X_2^1.
}
\]
Since
\[[h_2][\phi^3]\beta_1^3=[h_2][g_2]\beta_2^2=0,\]
we can lift $\beta_1^3$ along $[l_1^3]$ to $\pi\in[C,P]$. Since 
\[[g_2]\left(\beta_2^2-[l_2^2]\pi\right)=\beta_2^3-\beta_2^3=0,\]
we can lift $\beta_2^2-[l_2^2]\pi$ along $[f_2]$ to $\gamma_2^1\in[C,X_2^1]$. Let $\tilde{\pi}=\pi+[p_2^1]\gamma_2^1$. Then
\[[l_2^2]\tilde{\pi}=[l_2^2]\pi+[f_2]\gamma_2^1=\beta_2^2,\]
and
\[[l_1^3]\tilde{\pi}=[l_1^3]\pi+[l_1^3][p_2^1]\gamma_2^1=\beta_1^3.\]
Let $\beta_3^1=[l_3^1]\tilde{\pi}$, and we have
\[[\psi_3]\beta_3^1=[\psi_3][l_3^1]\tilde{\pi}=[g^2][l_2^2]\tilde{\pi}=\beta_3^2,\]
and
\[[h^1]\beta_3^1=[h^1][l_3^1]\tilde{\pi}=-[h_1][l_1^3]\tilde{\pi}=-[h_1]\beta_1^3.\]
Then the conclusion follows.
\end{proof}

\begin{rmk}
In this article, all spectral sequences are over $\F$. Therefore, we will ignore the sign from now on when applying Lemma \ref{lem6} and Corollary \ref{brprop}.
\end{rmk}

\section{The Borel and genuine $C_2$-equivariant Adams spectral sequences}\label{sec-genuine}
In this section, we will focus on the connection between the Borel and genuine $C_2$-equivariant Adams spectral sequences for $\SC$.

If we apply an exact functor to the $H_{\R}$-based Adams tower for $\SR$, we will get another tower, and another spectral sequence. In particular, if we apply $(-)^{\Theta}$, we will have:

\begin{lem}\label{ext0}
\[\Ext_{\A_{*,*}^{\R}}^{s,t,w}\left(\left(H_{\R}\right)^{\Theta}_{*,*}\right)=0.\]
\end{lem}

\begin{proof}
By the construction of the functors $(-)^{\Phi}$, $(-)^t$, $(-)^{\Theta}$, we can see that they are exact. Applying them to the $H_{\R}$-based Adams tower for $\SR$, we will get three spectral sequences. Since $H_{\R}\wedge H_{\R}$ splits as a wedge sum of suspensions of $H_{\R}$, their $E_1$-terms can be written as
\begin{gather*}
\pi^{\R}_{t-s,w}\left(H_{\R}\wedge\overline{H_{\R}}^{\wedge s}\right)^{\Phi}\cong C_{\A_{*,*}^{\R}}^{s,t,w}\left(\left(H_{\R}\right)^{\Phi}_{*,*}\right), \\
\pi^{\R}_{t-s,w}\left(H_{\R}\wedge\overline{H_{\R}}^{\wedge s}\right)^t\cong C_{\A_{*,*}^{\R}}^{s,t,w}\left(\left(H_{\R}\right)^t_{*,*}\right), \\
\pi^{\R}_{t-s,w}\left(H_{\R}\wedge\overline{H_{\R}}^{\wedge s}\right)^{\Theta}\cong C_{\A_{*,*}^{\R}}^{s,t,w}\left(\left(H_{\R}\right)^{\Theta}_{*,*}\right). \\
\end{gather*}
Note that the map $\left(H_{\R}\right)^{\Phi}_{*,*}\to\left(H_{\R}\right)^t_{*,*}$ is the standard inclusion $\F[\rho^{\pm},\tau]\to\F[\rho^{\pm},\tau^{\pm}]$, we have the short exact sequence
\begin{align*}
0 \to \pi^{\R}_{t-s,w}\left(H_{\R}\wedge\overline{H_{\R}}^{\wedge s}\right)^{\Phi} & \to \pi^{\R}_{t-s,w}\left(H_{\R}\wedge\overline{H_{\R}}^{\wedge s}\right)^t \\
& \to \pi^{\R}_{t-s,w}\left(H_{\R}\wedge\overline{H_{\R}}^{\wedge s}\right)^{\Theta} \to 0. 
\end{align*}
The $d_1$'s in these spectral sequences are isomorphic to the differentials in reduced cobar complexes, so they are compatible with the short exact sequences, and we have the long exact sequence of $E_2$-terms
\begin{align*}
\cdots & \to \Ext_{\A_{*,*}^{\R}}^{s,t,w}\left(\left(H_{\R}\right)^{\Phi}_{*,*}\right) \to \Ext_{\A_{*,*}^{\R}}^{s,t,w}\left(\left(H_{\R}\right)^t_{*,*}\right) \\
& \to \Ext_{\A_{*,*}^{\R}}^{s,t,w}\left(\left(H_{\R}\right)^{\Theta}_{*,*}\right) \to \Ext_{\A_{*,*}^{\R}}^{s+1,t,w}\left(\left(H_{\R}\right)^{\Phi}_{*,*}\right) \to \cdots.
\end{align*}
It suffices to show that the map
\begin{equation}\label{phit}
\Ext_{\A_{*,*}^{\R}}^{s,t,w}\left(\left(H_{\R}\right)^{\Phi}_{*,*}\right) \to \Ext_{\A_{*,*}^{\R}}^{s,t,w}\left(\left(H_{\R}\right)^t_{*,*}\right)
\end{equation}
is an isomorphism. Applying the change-of-rings isomorphism \cite[Prop.~1.4]{mr77} to the maps of Hopf algebroids
\[\left(\left(H_{\R}\right)_{*,*},\A_{*,*}^{\R}\right)\to\left(\left(H_{\R}\right)_{*,*}[\rho^{-1}],\A_{*,*}^{\R}[\rho^{-1}]\right)\]
and
\[\left(\left(H_{\R}\right)_{*,*},\A_{*,*}^{\R}\right)\to\left(\left(H_{\R}\right)_{*,*}[\rho^{-1},\tau^{-1}],\left(\A_{*,*}^{\R}\right)^{\wedge}_{\rho}[\rho^{-1},\tau^{-1}]\right),\]
we have
\[\Ext_{\A_{*,*}^{\R}}^{s,t,w}\left(\left(H_{\R}\right)_{*,*},\left(H_{\R}\right)^{\Phi}_{*,*}\right)\cong\Ext_{\A_{*,*}^{\R}[\rho^{-1}]}^{s,t,w}\left(\left(H_{\R}\right)_{*,*}[\rho^{-1}],\left(H_{\R}\right)_{*,*}[\rho^{-1}]\right),\]
and
\begin{align*}
&\Ext_{\A_{*,*}^{\R}}^{s,t,w}\left(\left(H_{\R}\right)_{*,*},\left(H_{\R}\right)^t_{*,*}\right)\\
\cong{}&\Ext_{\left(\A_{*,*}^{\R}\right)^{\wedge}_{\rho}[\rho^{-1},\tau^{-1}]}^{s,t,w}\left(\left(H_{\R}\right)_{*,*}[\rho^{-1},\tau^{-1}],\left(H_{\R}\right)_{*,*}[\rho^{-1},\tau^{-1}]\right).
\end{align*}
Then the map (\ref{phit}) is isomorphic to the map between the cohomology of Hopf algebroids, which is induced from the map of Hopf algebroids
\begin{equation}\label{hopf}
\left(\left(H_{\R}\right)_{*,*}[\rho^{-1}],\A_{*,*}^{\R}[\rho^{-1}]\right)\to\left(\left(H_{\R}\right)_{*,*}[\rho^{-1},\tau^{-1}],\left(\A_{*,*}^{\R}\right)^{\wedge}_{\rho}[\rho^{-1},\tau^{-1}]\right).
\end{equation}
By \cite[Thm.~4.1]{di}, the Hopf algebroids split as
\begin{align*}
& \left(\F[\rho^{\pm}],\F[\rho^{\pm}]\right)\otimes_{\F}\left(\F,\A_*''\right)\otimes_{\F}\left(\F[\tau],\F[\tau][x]\right)\\
\to & \left(\F[\rho^{\pm}],\F[\rho^{\pm}]\right)\otimes_{\F}\left(\F,\A_*''\right)\otimes_{\F}\left(\F[\tau^{\pm}],\F[\tau^{\pm}][x]^{\wedge}_x\right),
\end{align*}
where $x=\rho\tau_0$, and
\[\A_*''=\F[\xi_1,\xi_2,\cdots]\]
is isomorphic to the classical dual Steenrod algebra with degrees suitably shifted. \cite[Lem.~4.3]{di} shows that the cohomology of $\left(\F[\tau],\F[\tau][x]\right)$ is $\F$ concentrated in homological degree $0$, where they filtered the cobar complex by powers of $x$ and deduced the differentials $d_{2^i}(\tau^{2^i})=[x^{2^i}]$. The cohomology of $\left(\F[\tau^{\pm}],\F[\tau^{\pm}][x]^{\wedge}_x\right)$ can be similarly shown to be $\F$ concentrated in homological degree $0$ as well, where there are differentials $d_{2^i}(\tau^{-2^i})=\tau^{-2^{i+1}}[x^{2^i}]$ by the Leibniz's rule. Therefore, the map (\ref{hopf}) induces an isomorphism between their cohomology, which completes the proof.
\end{proof}

By \cite{GHIR} the $E_2$-term of the genuine $C_2$-equivariant Adams spectral sequence for $\SC$ splits as the direct sum of a positive cone part and a negative cone part:
\begin{align*}
&\Ext_{\A_{*,*}^{C_2}}^{s,t,w}\left(\left(H_{C_2}\right)_{*,*},\left(H_{C_2}\right)_{*,*}\right)\cong\Ext_{\A_{*,*}^{\R}}^{s,t,w}\left(\left(H_{\R}\right)_{*,*},\left(H_{\R}\right)^{C_2}_{*,*}\right)\\
\cong{}&\Ext_{\A_{*,*}^{\R}}^{s,t,w}\left(\left(H_{\R}\right)_{*,*},\left(H_{\R}\right)_{*,*}\right)\oplus\Ext_{\A_{*,*}^{\R}}^{s,t,w}\left(\left(H_{\R}\right)_{*,*},\left(H_{\R}\wedge \left(\SR\right)^{\Psi}\right)_{*,*}\right),
\end{align*}
where the first and second direct summand are denoted by $\Ext_{\R}$ and $\Ext_{NC}$ in \cite{GHIR}.
The following theorem indicates its relation to the $E_2$-term of the Borel Adams spectral sequence for $\SC$.

\begin{thm}\label{borext}
\begin{align*}
& \Ext_{\A_{*,*}^h}^{s,t,w}\left(\left(H_{C_2}\right)^h_{*,*},\left(H_{C_2}\right)^h_{*,*}\right) \\
\cong{} & H\left(\begin{gathered}\Ext_{\A_{*,*}^{\R}}^{s,t,w}\left(\left(H_{\R}\right)_{*,*},\left(H_{\R}\right)_{*,*}\right)\\\oplus\\\Ext_{\A_{*,*}^{\R}}^{s-1,t-1,w}\left(\left(H_{\R}\right)_{*,*},\left(H_{\R}\wedge \left(\SR\right)^{\Psi}\right)_{*,*}\right)\end{gathered},\delta\right),
\end{align*}
where the shortened differential $\delta$ is $0$ on the first summand, and is the composite map
\[
\xymatrix{
\Ext_{\A_{*,*}^{\R}}^{s-1,t-1,w}\left(\left(H_{\R}\wedge \left(\SR\right)^{\Psi}\right)_{*,*}\right)\ar[r] & \Ext_{\A_{*,*}^{\R}}^{s-1,t-1,w}\left(\left(H_{\R}\right)^{C_2}_{*,*}\right) \ar[ld]^{d_2^{C_2}} \\
\Ext_{\A_{*,*}^{\R}}^{s+1,t,w}\left(\left(H_{\R}\right)^{C_2}_{*,*}\right) \ar[r] & \Ext_{\A_{*,*}^{\R}}^{s+1,t,w}\left(\left(H_{\R}\right)_{*,*}\right) \\
}
\]
on the second summand.
\end{thm}

\begin{proof}
Applying Verdier's octahedral axiom to $X\to X^{C_2}\to X^h$, we have
\[
\xymatrix@!C{
X \ar[r] \ar[rd] & X^{C_2} \ar[rrr] \ar[d] & & & X\wedge\left(\SR\right)^{\Psi} \ar[lldd] \\
 & X^h \ar[rd] \ar[dd] & & & \\
 & & X^{\Psi} \ar[ld] & & \\
 & X^{\Theta}. & & & \\
}
\]
Applying these functors to the $H_{\R}$-based Adams tower for $\SR$, we will get several spectral sequences.

The fiber sequence
\[H_{\R}\wedge\overline{H_{\R}}^{\wedge s}\to\left(H_{\R}\wedge\overline{H_{\R}}^{\wedge s}\right)^h\to\left(H_{\R}\wedge\overline{H_{\R}}^{\wedge s}\right)^{\Psi}\]
induces a short exact sequence of $E_1$-terms
\[0\to\pi_{t-s,w}^{\R}\left(H_{\R}\wedge\overline{H_{\R}}^{\wedge s}\right)\to\pi_{t-s,w}^{\R}\left(H_{\R}\wedge\overline{H_{\R}}^{\wedge s}\right)^h\to\pi_{t-s,w}^{\R}\left(H_{\R}\wedge\overline{H_{\R}}^{\wedge s}\right)^{\Psi}\to0,\]
which induces a long exact sequence of $E_2$-terms
\begin{align*}
\cdots & \to \Ext_{\A_{*,*}^{\R}}^{s,t,w}\left(\left(H_{\R}\right)_{*,*}\right) \to \Ext_{\A_{*,*}^{\R}}^{s,t,w}\left(\left(H_{\R}\right)^h_{*,*}\right) \\
& \to \Ext_{\A_{*,*}^{\R}}^{s,t,w}\left(\left(H_{\R}\right)^{\Psi}_{*,*}\right) \to \Ext_{\A_{*,*}^{\R}}^{s+1,t,w}\left(\left(H_{\R}\right)_{*,*}\right) \to \cdots.
\end{align*}
The fiber sequence
\[H_{\R}\wedge\overline{H_{\R}}^{\wedge s}\wedge\left(\SR\right)^{\Psi}\to\left(H_{\R}\wedge\overline{H_{\R}}^{\wedge s}\right)^{\Psi}\to\left(H_{\R}\wedge\overline{H_{\R}}^{\wedge s}\right)^{\Theta}\]
induces a short exact sequence of the $E_1$-terms
\begin{align*}
0\to\pi_{t-s,w}^{\R}\left(H_{\R}\wedge\overline{H_{\R}}^{\wedge s}\right)^{\Psi} & \to\pi_{t-s,w}^{\R}\left(H_{\R}\wedge\overline{H_{\R}}^{\wedge s}\right)^{\Theta}\\
& \to\pi_{t-s-1,w}^{\R}\left(H_{\R}\wedge\overline{H_{\R}}^{\wedge s}\wedge\left(\SR\right)^{\Psi}\right)\to0,
\end{align*}
which induces a long exact sequence of the $E_2$-terms
\begin{align*}
\cdots & \to \Ext_{\A_{*,*}^{\R}}^{s,t,w}\left(\left(H_{\R}\right)^{\Psi}_{*,*}\right) \to \Ext_{\A_{*,*}^{\R}}^{s,t,w}\left(\left(H_{\R}\right)^{\Theta}_{*,*}\right) \\
& \to \Ext_{\A_{*,*}^{\R}}^{s,t-1,w}\left(\left(H_{\R}\wedge \left(\SR\right)^{\Psi}\right)_{*,*}\right) \to \Ext_{\A_{*,*}^{\R}}^{s+1,t,w}\left(\left(H_{\R}\right)^{\Psi}_{*,*}\right) \to \cdots.
\end{align*}
By Lemma \ref{ext0}, $\Ext_{\A_{*,*}^{\R}}^{*,*,*}((H_{\R})^{\Theta}_{*,*})=0$, so the boundary map 
\[\Ext_{\A_{*,*}^{\R}}^{s,t-1,w}\left(\left(H_{\R}\wedge \left(\SR\right)^{\Psi}\right)_{*,*}\right) \to \Ext_{\A_{*,*}^{\R}}^{s+1,t,w}\left(\left(H_{\R}\right)^{\Psi}_{*,*}\right)\]
is an isomorphism. Therefore, it suffices to show that the given $\delta$ coincides with the composite map $\delta'$:
\begin{equation}\label{delta}
\begin{aligned}
\xymatrix{
\Ext_{\A_{*,*}^{\R}}^{s-1,t-1,w}\left(\left(H_{\R}\wedge \left(\SR\right)^{\Psi}\right)_{*,*}\right) \ar[r] \ar[rd]_{\delta'} & \Ext_{\A_{*,*}^{\R}}^{s,t,w}\left(\left(H_{\R}\right)^{\Psi}_{*,*}\right) \ar[d] \\
 & \Ext_{\A_{*,*}^{\R}}^{s+1,t,w}\left(\left(H_{\R}\right)_{*,*}\right).
}
\end{aligned}
\end{equation}
The proof for this is essentially the same as the proof of \cite[Thm.~1.1]{br}, where \cite[Prop.~2.3]{br} is generalized to Corollary \ref{brprop} to fit into our case. Here we give a sketch:

In Corollary \ref{brprop}, set up the diagram as below:
\[
\xymatrix@C=.5cm{
\left(\frac{\overline{H_{\R}}^{\wedge (s+1)}}{\overline{H_{\R}}^{\wedge (s+2)}}\right)^h \ar[r]^{f} \ar[d]^{i} & \left(\frac{\overline{H_{\R}}^{\wedge s}}{\overline{H_{\R}}^{\wedge (s+2)}}\right)^h \ar[r]^{g} \ar[d]^{j} & \left(\frac{\overline{H_{\R}}^{\wedge s}}{\overline{H_{\R}}^{\wedge (s+1)}}\right)^h \ar[r]^(0.45){h} \ar[d]^{k} & \Sigma^{1,0}\left(\frac{\overline{H_{\R}}^{\wedge (s+1)}}{\overline{H_{\R}}^{\wedge (s+2)}}\right)^h \ar[d]^{i} \\
\left(\frac{\overline{H_{\R}}^{\wedge (s+1)}}{\overline{H_{\R}}^{\wedge (s+2)}}\right)^{\Theta} \ar[r]^{f'} \ar[d]^{i'} & \left(\frac{\overline{H_{\R}}^{\wedge s}}{\overline{H_{\R}}^{\wedge (s+2)}}\right)^{\Theta} \ar[r]^{g'} \ar[d]^{j'} & \left(\frac{\overline{H_{\R}}^{\wedge s}}{\overline{H_{\R}}^{\wedge (s+1)}}\right)^{\Theta} \ar[r]^(0.45){h'} & \Sigma^{1,0}\left(\frac{\overline{H_{\R}}^{\wedge (s+1)}}{\overline{H_{\R}}^{\wedge (s+2)}}\right)^{\Theta}\\
\Sigma^{1,0}\left(\frac{\overline{H_{\R}}^{\wedge (s+1)}}{\overline{H_{\R}}^{\wedge (s+2)}}\right)^{C_2} \ar[r]^{f''} \ar[d]^{i''} & \Sigma^{1,0}\left(\frac{\overline{H_{\R}}^{\wedge s}}{\overline{H_{\R}}^{\wedge (s+2)}}\right)^{C_2} \ar[d]^{j''} \\
\Sigma^{1,0}\left(\frac{\overline{H_{\R}}^{\wedge (s+1)}}{\overline{H_{\R}}^{\wedge (s+2)}}\right)^h \ar[r]^{f} & \Sigma^{1,0}\left(\frac{\overline{H_{\R}}^{\wedge s}}{\overline{H_{\R}}^{\wedge (s+2)}}\right)^h. \\
}
\]
Let $C$ be $S_{\R}^{t-s,w}$. Let 
\[x\in\pi_{t-s,w}^{\R}\left(\overline{H_{\R}}^{\wedge (s-1)}\middle/\overline{H_{\R}}^{\wedge s}\right)\wedge\left(\SR\right)^{\Psi}\]
represent the element
\[[x]\in\Ext_{\A_{*,*}^{\R}}^{s-1,t-1,w}\left(\left(H_{\R}\wedge \left(\SR\right)^{\Psi}\right)_{*,*}\right).\]
Consider the diagram
\[
\xymatrix{
\pi^{\R}_{t-s+1,w}\left(\frac{\overline{H_{\R}}^{\wedge (s-1)}}{\overline{H_{\R}}^{\wedge s}}\right)^{\Theta} \ar^{\beta}[r] \ar@{->>}_{\alpha_1}[d] & \pi^{\R}_{t-s,w}\left(\frac{\overline{H_{\R}}^{\wedge s}}{\overline{H_{\R}}^{\wedge (s+2)}}\right)^{\Theta} \ar^{[C,j']}[dd] \\
\pi^{\R}_{t-s,w}\left(\frac{\overline{H_{\R}}^{\wedge (s-1)}}{\overline{H_{\R}}^{\wedge s}}\right)\wedge\left(\SR\right)^{\Psi} \ar@{^(->}_{\alpha_2}[d] & \\
\pi^{\R}_{t-s,w}\left(\frac{\overline{H_{\R}}^{\wedge (s-1)}}{\overline{H_{\R}}^{\wedge s}}\right)^{C_2} \ar^{\beta'}[r] & \pi^{\R}_{t-s-1,w}\left(\frac{\overline{H_{\R}}^{\wedge s}}{\overline{H_{\R}}^{\wedge (s+2)}}\right)^{C_2}. \\
}
\]
Let $\omega=\beta'\alpha_2x\in[C,X_3^2]$. Then $d_2^{C_2}([\alpha_2x])$ can be represented by an arbitrary element in $(f'')^{-1}(\omega)$. Given the factorizations
\[
\xymatrix@C=0.48cm{
[C,i'']:\pi^{\R}_{t-s-1,w}\left(\frac{\overline{H_{\R}}^{\wedge (s+1)}}{\overline{H_{\R}}^{\wedge (s+2)}}\right)^{C_2}\ar@{->>}^(0.59){i_1''}[r] & \pi^{\R}_{t-s-1,w}\left(\frac{\overline{H_{\R}}^{\wedge (s+1)}}{\overline{H_{\R}}^{\wedge (s+2)}}\right)\ar@{^(->}^(0.48){i_2''}[r] & \pi^{\R}_{t-s-1,w}\left(\frac{\overline{H_{\R}}^{\wedge (s+1)}}{\overline{H_{\R}}^{\wedge (s+2)}}\right)^h, \\
}
\]
and
\[
\xymatrix{
[C,k]:\pi^{\R}_{t-s,w}\left(\frac{\overline{H_{\R}}^{\wedge s}}{\overline{H_{\R}}^{\wedge (s+1)}}\right)^h\ar@{->>}^(0.55){k_1}[r] & \pi^{\R}_{t-s,w}\left(\frac{\overline{H_{\R}}^{\wedge s}}{\overline{H_{\R}}^{\wedge (s+1)}}\right)^{\Psi}\ar@{^(->}^(0.49){k_2}[r] & \pi^{\R}_{t-s,w}\left(\frac{\overline{H_{\R}}^{\wedge s}}{\overline{H_{\R}}^{\wedge (s+1)}}\right)^h, \\
}
\]
we can see that $\delta([x])$ is represented by any element in $i_1''(f'')^{-1}(\omega)$. On the other hand, by the construction of the boundary maps, any element in
\[\left(k_2\right)^{-1}g'\beta\left(\alpha_1\right)^{-1}(x)\subset\left(k_2\right)^{-1}g'\left(j'\right)^{-1}(\omega)\]
can represent the image of $[x]$ under the boundary map
\[\Ext_{\A_{*,*}^{\R}}^{s-1,t-1,w}\left(\left(H_{\R}\wedge \left(\SR\right)^{\Psi}\right)_{*,*}\right)\to\Ext_{\A_{*,*}^{\R}}^{s,t,w}\left(\left(H_{\R}\right)^{\Psi}_{*,*}\right).\]
Then an element $y\in\left(i_2''\right)^{-1}hk^{-1}g'\left(j'\right)^{-1}(\omega)$ representing $\delta'([x])$. By Corollary \ref{brprop},
\[i_2''y\in hk^{-1}g'\left(j'\right)^{-1}(\omega)=i''\left(f''\right)^{-1}(\omega).\]
By the injectivity of $i_2''$, $y\in i_1''(f'')^{-1}(\omega)$, and hence represent $\delta([x])$ as well. Therefore, $\delta'$ and $\delta$ are the same map.
\end{proof}

\begin{rmk}
The composite map $d$ is not trivial. For example, 
\[\Ext_{\A_{*,*}^{\R}}^{9,23,11}\left(\left(H_{\R}\right)^h_{*,*}\right)=0,\]
while
\[\Ext_{\A_{*,*}^{\R}}^{9,23,11}\left(\left(H_{\R}\right)_{*,*}\right)=\F\left\{h_1^6c_0\right\}\]
is nontrivial (c.f. charts in Section \ref{sec-names} and \cite{gi}). This implies that $h_1^6c_0$ must be a target under $d$. In fact, from the charts in \cite{gi}, we can see that 
\[\Ext_{\A_{*,*}^{\R}}^{7,22,11}\left(\left(H_{\R}\wedge\left(\SR\right)^{\Psi}\right)_{*,*}\right)=\F\left\{\rho^{-2}QPh_1^4\right\},\]
and thus $\delta(\rho^{-2}QPh_1^4)=h_1^6c_0$. This can also be seen from \cite[Table~11.3]{gi}.
\end{rmk}

The theorem above indicates that the $d_2$'s in the genuine $C_2$-equivariant Adams spectral sequence connecting the negative cone part and the positive cone part are shortened in the Borel Adams spectral sequence, such that the sources and targets are cancelled before the $E_2$-page. These differentials will be called shortened differentials.

Consider the maps from a fiber sequence
\[X^h\xrightarrow{i}X^{\Theta}\xrightarrow{p}\Sigma^{1,0}X^{C_2}\xrightarrow{j}\Sigma^{1,0}X^h.\]
The following theorem describes the phenomenon of shortened differentials in higher pages.

\begin{thm}\label{shortd}
Let $z\in\pi^{\R}_{t-s,w}\left(\overline{H_{\R}}^{\wedge (s-1)}\middle/\overline{H_{\R}}^{\wedge s}\right)^{C_2}$ be an element with $j_*z=0$. Suppose $z$ supports a nontrivial differential $d_r^{C_2}([z])=[z']$, where $r\ge3$ and $z'\in E_1^{C_2}$. Then there exists an element $x\in\pi^{\R}_{t-s,w}\left(\overline{H_{\R}}^{\wedge s}\middle/\overline{H_{\R}}^{\wedge (s+1)}\right)^h$ surviving to $E_2^h$, satisfying that the image of $[x]\in E_2^h$ under the map
\[\Ext_{\A_{*,*}^{\R}}^{*,*,*}\left(\left(H_{\R}\right)^h_{*,*}\right)\to\Ext_{\A_{*,*}^{\R}}^{*,*,*}\left(\left(H_{\R}\right)^{\Psi}_{*,*}\right)\cong\Ext_{\A_{*,*}^{\R}}^{*-1,*-1,*}\left(\left(H_{\R}\wedge\left(\SR\right)^{\Psi}\right)_{*,*}\right)\]
coincides with $[z]$, and that
\[d_{r-1}^h\left([x]\right)=[j_*z'].\]
\end{thm}

The claimed differentials will also be referred to as the shortened differentials.

\begin{proof}
Since $z$ supports a nontrivial differential $d_r^{C_2}$, it can be lifted to the elements $\beta\in\pi^{\R}_{t-s,w}\left(\overline{H_{\R}}^{\wedge (s-1)}\middle/\overline{H_{\R}}^{\wedge (s+1)}\right)^{C_2}$ and $\tilde{\beta}\in\pi^{\R}_{t-s,w}\left(\overline{H_{\R}}^{\wedge (s-1)}\middle/\overline{H_{\R}}^{\wedge (s+r)}\right)^{C_2}$. Then we can apply Lemma \ref{lem6} to the following diagram:
\[
\xymatrix{
\Sigma^{-1,0}\left(\frac{\overline{H_{\R}}^{\wedge (s-1)}}{\overline{H_{\R}}^{\wedge (s+1)}}\right)^h \ar[r] \ar[d] & \Sigma^{-1,0}\left(\frac{\overline{H_{\R}}^{\wedge (s-1)}}{\overline{H_{\R}}^{\wedge s}}\right)^h \ar[r] \ar[d] & \left(\frac{\overline{H_{\R}}^{\wedge s}}{\overline{H_{\R}}^{\wedge (s+1)}}\right)^h \ar[r] \ar[d]^i & \left(\frac{\overline{H_{\R}}^{\wedge (s-1)}}{\overline{H_{\R}}^{\wedge (s+1)}}\right)^h \ar[d] \\
\Sigma^{-1,0}\left(\frac{\overline{H_{\R}}^{\wedge (s-1)}}{\overline{H_{\R}}^{\wedge (s+1)}}\right)^{\Theta} \ar[r] \ar[d] & \Sigma^{-1,0}\left(\frac{\overline{H_{\R}}^{\wedge (s-1)}}{\overline{H_{\R}}^{\wedge s}}\right)^{\Theta} \ar[r]^(0.58){d_1^{\Theta}} \ar[d]^p & \left(\frac{\overline{H_{\R}}^{\wedge s}}{\overline{H_{\R}}^{\wedge (s+1)}}\right)^{\Theta} \ar[r] & \left(\frac{\overline{H_{\R}}^{\wedge (s-1)}}{\overline{H_{\R}}^{\wedge (s+1)}}\right)^{\Theta} \\
\left(\frac{\overline{H_{\R}}^{\wedge (s-1)}}{\overline{H_{\R}}^{\wedge (s+1)}}\right)^{C_2} \ar[r] \ar[d]^j & \left(\frac{\overline{H_{\R}}^{\wedge (s-1)}}{\overline{H_{\R}}^{\wedge s}}\right)^{C_2} \ar[d] \\
\left(\frac{\overline{H_{\R}}^{\wedge (s-1)}}{\overline{H_{\R}}^{\wedge (s+1)}}\right)^h \ar[r] & \left(\frac{\overline{H_{\R}}^{\wedge (s-1)}}{\overline{H_{\R}}^{\wedge s}}\right)^h. \\
}
\]
Let $y$ be the chosen lift of $z$ along $p_*$, and $x$ be the chosen lift of $d_1^{\Theta}(y)$ along $i_*$. Then we have the following diagram:
\[
\xymatrix@C=1cm{
j_*z' & \pi^{\R}_{t-s-1,w}\left(\frac{\overline{H_{\R}}^{\wedge (s+r-1)}}{\overline{H_{\R}}^{\wedge (s+r)}}\right)^h & \pi^{\R}_{t-s-1,w}\left(\frac{\overline{H_{\R}}^{\wedge (s+r-1)}}{\overline{H_{\R}}^{\wedge (s+r)}}\right)^{C_2}\ar[l]_{j_*} & z' \\
j_*\tilde{\beta} & \pi^{\R}_{t-s,w}\left(\frac{\overline{H_{\R}}^{\wedge (s-1)}}{\overline{H_{\R}}^{\wedge (s+r-1)}}\right)^h \ar[u]\ar[d] & \pi^{\R}_{t-s,w}\left(\frac{\overline{H_{\R}}^{\wedge (s-1)}}{\overline{H_{\R}}^{\wedge (s+r-1)}}\right)^{C_2} \ar[l]_{j_*} \ar[u]\ar[d] & \tilde{\beta} \\
j_*\beta & \pi^{\R}_{t-s,w}\left(\frac{\overline{H_{\R}}^{\wedge (s-1)}}{\overline{H_{\R}}^{\wedge (s+1)}}\right)^h & \pi^{\R}_{t-s,w}\left(\frac{\overline{H_{\R}}^{\wedge (s-1)}}{\overline{H_{\R}}^{\wedge (s+1)}}\right)^{C_2} \ar[l]_{j_*} \ar[d] & \beta \\
x & \pi^{\R}_{t-s,w}\left(\frac{\overline{H_{\R}}^{\wedge s}}{\overline{H_{\R}}^{\wedge (s+1)}}\right)^h \ar@{-->}@/^5.5pc/[uuu]^{d_{r-1}^h} \ar[u]\ar[d]_{i_*} \ar@{}[r]|*++=[o][F]{\txt{Lemma\\\phantom{i}\ref{lem6}\phantom{i}}} & \pi^{\R}_{t-s,w}\left(\frac{\overline{H_{\R}}^{\wedge (s-1)}}{\overline{H_{\R}}^{\wedge s}}\right)^{C_2} \ar@{-->}@/_5.5pc/[uuu]_{d_r^{C_2}} & z \\
d_1^{\Theta}(y) & \pi^{\R}_{t-s,w}\left(\frac{\overline{H_{\R}}^{\wedge s}}{\overline{H_{\R}}^{\wedge (s+1)}}\right)^{\Theta} & \pi^{\R}_{t-s+1,w}\left(\frac{\overline{H_{\R}}^{\wedge (s-1)}}{\overline{H_{\R}}^{\wedge s}}\right)^{\Theta}. \ar[l]^{d_1^{\Theta}} \ar[u]_{p_*} & y \\
}
\]
Therefore, $d_{r-1}^h\left([x]\right)=[j_*z']$. On the other hand, consider the diagram 
\[
\xymatrix{
x & \pi^{\R}_{t-s,w}\left(\frac{\overline{H_{\R}}^{\wedge s}}{\overline{H_{\R}}^{\wedge (s+1)}}\right)^h \ar[r]\ar[rd]^{i_*} & \pi^{\R}_{t-s,w}\left(\frac{\overline{H_{\R}}^{\wedge s}}{\overline{H_{\R}}^{\wedge (s+1)}}\right)^{\Psi} \ar[d] & \\
y & \pi^{\R}_{t-s+1,w}\left(\frac{\overline{H_{\R}}^{\wedge s}}{\overline{H_{\R}}^{\wedge (s+1)}}\right)^{\Theta} \ar[r]^{d_1^{\Theta}} \ar[rd] \ar[d]_{p_*} & \pi^{\R}_{t-s,w}\left(\frac{\overline{H_{\R}}^{\wedge s}}{\overline{H_{\R}}^{\wedge (s+1)}}\right)^{\Theta} & i_*x \\
z & \pi^{\R}_{t-s,w}\left(\frac{\overline{H_{\R}}^{\wedge (s-1)}}{\overline{H_{\R}}^{\wedge s}}\right)^{C_2} \ar[r] & \pi^{\R}_{t-s,w}\left(\frac{\overline{H_{\R}}^{\wedge (s-1)}}{\overline{H_{\R}}^{\wedge s}}\wedge\left(\SR\right)^{\Psi}\right), \ar@{-->}@/_5pc/[uu] & \left([z]\right) \\
}
\]
where the dashed line induces the boundary map between the $\Ext$ groups. This implies that $[x]$ maps to $[z]$, and the result follows.
\end{proof}

Note that since
\[j_*:\pi^{\R}_{t-s-1,w}\left(\overline{H_{\R}}^{\wedge (s+r)}\middle/\overline{H_{\R}}^{\wedge (s+r+1)}\right)^{C_2}\to\pi^{\R}_{t-s-1,w}\left(\overline{H_{\R}}^{\wedge (s+r)}\middle/\overline{H_{\R}}^{\wedge (s+r+1)}\right)^h\]
factors through $\pi^{\R}_{t-s-1,w}\left(\overline{H_{\R}}^{\wedge (s+r)}\middle/\overline{H_{\R}}^{\wedge (s+r+1)}\right)$, the class $j_*z'$ is the positive cone part of $z'$.

If $j_*z'=0$, then $d_r^{C_2}\left([z]\right)$ lies in the negative cone part. The following theorem will show that there are differentials of the same length in the Borel and geunine $C_2$-equivariant Adams spectral sequences, such that the negative cone parts of their sources and targets are the same respectively.

\begin{thm}\label{diffnc}
Let $z\in\pi^{\R}_{t-s,w}\left(\overline{H_{\R}}^{\wedge (s-1)}\middle/\overline{H_{\R}}^{\wedge s}\right)^{C_2}$ be an element with $j_*z=0$. Suppose $z$ supports a nontrivial differential $d_r^{C_2}([z])=[z']$, where $r\ge2$ and $z'\in E_1^{C_2}$. Suppose further that $j_*z'=0$. Let $x\in\pi^{\R}_{t-s,w}\left(\overline{H_{\R}}^{\wedge s}\middle/\overline{H_{\R}}^{\wedge (s+1)}\right)^h$ be as in Theorem \ref{shortd}. Then there exists an element $x'\in\pi^{\R}_{t-s-1,w}\left(\overline{H_{\R}}^{\wedge (s+r)}\middle/\overline{H_{\R}}^{\wedge (s+r+1)}\right)^h$ surviving to $E_r^h$, satisfying that the image of $[x']\in E_2^h$ under the map
\[\Ext_{\A_{*,*}^{\R}}^{*,*,*}\left(\left(H_{\R}\right)^h_{*,*}\right)\to\Ext_{\A_{*,*}^{\R}}^{*,*,*}\left(\left(H_{\R}\right)^{\Psi}_{*,*}\right)\cong\Ext_{\A_{*,*}^{\R}}^{*-1,*-1,*}\left(\left(H_{\R}\wedge\left(\SR\right)^{\Psi}\right)_{*,*}\right)\]
coincides with $[z']$, and that
\[d_r^h\left([x]\right)=[x'].\]
\end{thm} 

\begin{proof}
Let $x$, $y$, $\beta$, and $\tilde{\beta}$ be the same elements as appeared in the proof of Theorem \ref{shortd}. Let $\beta'\in\pi_{t-s-1,w}^{\R}\left(\overline{H_{\R}}^{\wedge (s+r-1)}\middle/\overline{H_{\R}}^{\wedge (s+r+1)}\right)^{C_2}$ be the image of $\tilde{\beta}$, which maps to the element $z'$. Since $j_*z'=0$, we can apply Lemma \ref{lem6} to the following diagram:
\[
\xymatrix@C=.4cm{
\Sigma^{-1,0}\left(\frac{\overline{H_{\R}}^{\wedge (s+r-1)}}{\overline{H_{\R}}^{\wedge (s+r+1)}}\right)^h \ar[r] \ar[d] & \Sigma^{-1,0}\left(\frac{\overline{H_{\R}}^{\wedge (s+r-1)}}{\overline{H_{\R}}^{\wedge (s+r)}}\right)^h \ar[r] \ar[d] & \left(\frac{\overline{H_{\R}}^{\wedge (s+r)}}{\overline{H_{\R}}^{\wedge (s+r+1)}}\right)^h \ar[r] \ar[d]^i & \left(\frac{\overline{H_{\R}}^{\wedge (s+r-1)}}{\overline{H_{\R}}^{\wedge (s+r+1)}}\right)^h \ar[d] \\
\Sigma^{-1,0}\left(\frac{\overline{H_{\R}}^{\wedge (s+r-1)}}{\overline{H_{\R}}^{\wedge (s+r+1)}}\right)^{\Theta} \ar[r] \ar[d] & \Sigma^{-1,0}\left(\frac{\overline{H_{\R}}^{\wedge (s+r-1)}}{\overline{H_{\R}}^{\wedge (s+r)}}\right)^{\Theta} \ar[r]^(0.57){d_1^{\Theta}} \ar[d]^p & \left(\frac{\overline{H_{\R}}^{\wedge (s+r)}}{\overline{H_{\R}}^{\wedge (s+r+1)}}\right)^{\Theta} \ar[r] & \left(\frac{\overline{H_{\R}}^{\wedge (s+r-1)}}{\overline{H_{\R}}^{\wedge (s+r+1)}}\right)^{\Theta} \\
\left(\frac{\overline{H_{\R}}^{\wedge (s+r-1)}}{\overline{H_{\R}}^{\wedge (s+r+1)}}\right)^{C_2} \ar[r] \ar[d]^j & \left(\frac{\overline{H_{\R}}^{\wedge (s+r-1)}}{\overline{H_{\R}}^{\wedge (s+r)}}\right)^{C_2} \ar[d] \\
\left(\frac{\overline{H_{\R}}^{\wedge (s+r-1)}}{\overline{H_{\R}}^{\wedge (s+r+1)}}\right)^h \ar[r] & \left(\frac{\overline{H_{\R}}^{\wedge (s+r-1)}}{\overline{H_{\R}}^{\wedge (s+r)}}\right)^h. \\
}
\]
Let $y'$ be the chosen lift of $z'$ along $p_*$, and $x'$ be the chosen lift of $d_1^{\Theta}(y')$ along $i_*$. Then we have the following diagram:
\[
\xymatrix@C=.3cm{
d_1^{\Theta}(y') & & \pi^{\R}_{t-s-1,w}\left(\frac{\overline{H_{\R}}^{\wedge (s+r)}}{\overline{H_{\R}}^{\wedge (s+r+1)}}\right)^{\Theta} & & & \pi^{\R}_{t-s-1,w}\left(\frac{\overline{H_{\R}}^{\wedge (s+r-1)}}{\overline{H_{\R}}^{\wedge (s+r)}}\right)^{\Theta} \ar[lll]_{d_1^{\Theta}} \ar[d]^{p_*} & & y' \\
x' & & \pi^{\R}_{t-s-1,w}\left(\frac{\overline{H_{\R}}^{\wedge (s+r)}}{\overline{H_{\R}}^{\wedge (s+r+1)}}\right)^h \ar[u]^{i_*} \ar[d] \ar@{}[rrr]|*++=[o][F]{\txt{Lemma\\\phantom{i}\ref{lem6}\phantom{i}}} & & & \pi^{\R}_{t-s-1,w}\left(\frac{\overline{H_{\R}}^{\wedge (s+r-1)}}{\overline{H_{\R}}^{\wedge (s+r)}}\right)^{C_2} & & z' \\
j_*\beta' & & \pi^{\R}_{t-s-1,w}\left(\frac{\overline{H_{\R}}^{\wedge (s+r-1)}}{\overline{H_{\R}}^{\wedge (s+r+1)}}\right)^h & & & \pi^{\R}_{t-s-1,w}\left(\frac{\overline{H_{\R}}^{\wedge (s+r-1)}}{\overline{H_{\R}}^{\wedge (s+r+1)}}\right)^{C_2} \ar[lll]^{j_*} \ar[u] & & \beta' \\
j_*\tilde{\beta} & & \pi^{\R}_{t-s,w}\left(\frac{\overline{H_{\R}}^{\wedge (s-1)}}{\overline{H_{\R}}^{\wedge (s+r-1)}}\right)^h \ar[u]\ar[d] & & & \pi^{\R}_{t-s,w}\left(\frac{\overline{H_{\R}}^{\wedge (s-1)}}{\overline{H_{\R}}^{\wedge (s+r-1)}}\right)^{C_2} \ar[lll]_{j_*} \ar[u]\ar[d] & & \tilde{\beta} \\
j_*\beta & & \pi^{\R}_{t-s,w}\left(\frac{\overline{H_{\R}}^{\wedge (s-1)}}{\overline{H_{\R}}^{\wedge (s+1)}}\right)^h & & & \pi^{\R}_{t-s,w}\left(\frac{\overline{H_{\R}}^{\wedge (s-1)}}{\overline{H_{\R}}^{\wedge (s+1)}}\right)^{C_2} \ar[lll]_{j_*} \ar[d] & & \beta \\
x & & \pi^{\R}_{t-s,w}\left(\frac{\overline{H_{\R}}^{\wedge s}}{\overline{H_{\R}}^{\wedge (s+1)}}\right)^h \ar@{-->}`l_u/4pt[l]`/4pt[uuuu]_(0.4){d_r^h}[uuuu] \ar[u]\ar[d]_{i_*} \ar@{}[rrr]|*++=[o][F]{\txt{Lemma\\\phantom{i}\ref{lem6}\phantom{i}}} & & & \pi^{\R}_{t-s,w}\left(\frac{\overline{H_{\R}}^{\wedge (s-1)}}{\overline{H_{\R}}^{\wedge s}}\right)^{C_2} \ar@{-->}`r^u/4pt[r]`/4pt[uuuu]^(0.4){d_r^{C_2}}[uuuu] & & z \\
d_1^{\Theta}(y) & & \pi^{\R}_{t-s,w}\left(\frac{\overline{H_{\R}}^{\wedge s}}{\overline{H_{\R}}^{\wedge (s+1)}}\right)^{\Theta} & & & \pi^{\R}_{t-s+1,w}\left(\frac{\overline{H_{\R}}^{\wedge (s-1)}}{\overline{H_{\R}}^{\wedge s}}\right)^{\Theta} \ar[lll]^{d_1^{\Theta}}. \ar[u]_{p_*} & & y \\
}
\]
From the diagram, we can see that $d_r^h\left([x']\right)=[x'']$. By the similar analysis as in the proof of Theorem \ref{shortd}, we can see that $x''$ represents the class whose image is $[z]$ under the boundary map of $\Ext$ groups, which completes the proof.
\end{proof}

The following theorem characterizes the relative behavior in the Borel Adams spectral sequence when an element in the negative cone part of $E_2^{C_2}$ survives to $E_{\infty}^{C_2}$.

\begin{thm}
Let $z\in\pi^{\R}_{t-s,w}\left(\overline{H_{\R}}^{\wedge (s-1)}\middle/\overline{H_{\R}}^{\wedge s}\right)^{C_2}$ be an element with $j_*z=0$. Suppose $z$ survives to $E_{\infty}^{C_2}$ and detects $\hat{z}\in\pi^{\R}_{t-s-1,w}\left(\SR\right)^{C_2}$. Then there exists an element $x\in\pi^{\R}_{t-s,w}\left(\overline{H_{\R}}^{\wedge s}\middle/\overline{H_{\R}}^{\wedge (s+1)}\right)^h$ surviving to $E_{\infty}^h$ and detecting $j_*\hat{z}$, such that the image of $[x]\in E_2^h$ under the map
\[\Ext_{\A_{*,*}^{\R}}^{*,*,*}\left(\left(H_{\R}\right)^h_{*,*}\right)\to\Ext_{\A_{*,*}^{\R}}^{*,*,*}\left(\left(H_{\R}\right)^{\Psi}_{*,*}\right)\cong\Ext_{\A_{*,*}^{\R}}^{*-1,*-1,*}\left(\left(H_{\R}\wedge\left(\SR\right)^{\Psi}\right)_{*,*}\right)\]
coincides with $[z]$.
\end{thm}

\begin{proof}
Let $x$, $y$, and $\beta$ be the same elements as appeared in the proof of Theorem \ref{shortd}. Consider the following diagram:
\[
\xymatrix@C=.3cm{
j_*\hat{z} & \pi_{t-s,w}^{\R}\left(\SR\right)^h & & \pi_{t-s,w}^{\R}\left(\SR\right)^{C_2} \ar[ll]_{j_*} & \hat{z} \\
\txt{$j_*\hat{\beta}$\\$r_*\hat{\alpha}$} & \pi_{t-s,w}^{\R}\left(\overline{H_{\R}}^{\wedge (s-1)}\right)^h \ar[u]\ar[rd] & & \pi_{t-s,w}^{\R}\left(\overline{H_{\R}}^{\wedge (s-1)}\right)^{C_2} \ar[ll]_{j_*} \ar[u]\ar[d] & \hat{\beta} \\
\hat{\alpha} & \pi_{t-s,w}^{\R}\left(\overline{H_{\R}}^{\wedge s}\right)^h \ar[u]^{r_*}\ar[d] & \pi_{t-s,w}^{\R}\left(\frac{\overline{H_{\R}}^{\wedge (s-1)}}{\overline{H_{\R}}^{\wedge (s+1)}}\right)^h & \pi_{t-s,w}^{\R}\left(\frac{\overline{H_{\R}}^{\wedge (s-1)}}{\overline{H_{\R}}^{\wedge (s+1)}}\right)^{C_2} \ar[l]_{j_*} \ar[d] & \beta \\
x & \pi_{t-s,w}^{\R}\left(\frac{\overline{H_{\R}}^{\wedge s}}{\overline{H_{\R}}^{\wedge (s+1)}}\right)^h \ar[ru]\ar[d]_{i_*} \ar@{}[rr]|*++=[o][F]{\txt{Lemma\\\phantom{i}\ref{lem6}\phantom{i}}} & & \pi_{t-s,w}^{\R}\left(\frac{\overline{H_{\R}}^{\wedge (s-1)}}{\overline{H_{\R}}^{\wedge s}}\right)^{C_2} & z \\
d_1^{\Theta}(y) & \pi_{t-s,w}^{\R}\left(\frac{\overline{H_{\R}}^{\wedge s}}{\overline{H_{\R}}^{\wedge (s+1)}}\right)^{\Theta} & & \pi_{t-s+1,w}^{\R}\left(\frac{\overline{H_{\R}}^{\wedge (s-1)}}{\overline{H_{\R}}^{\wedge s}}\right)^{\Theta}. \ar[ll]^{d_1^{\Theta}} \ar[u]_{p_*} & y \\
}
\]
By our assumption, $\beta$ can be lifted to $\hat{\beta}\in\pi_{t-s,w}^{\R}\left(\overline{H_{\R}}^{\wedge (s-1)}\right)^{C_2}$. Then $j_*\beta$ is the image of $j_*\hat{\beta}$ and thus maps to $0$ in $\pi_{t-s-1,w}^{\R}\left(\overline{H_{\R}}^{\wedge (s+1)}\right)^h$. So does $x$ since $x$ maps to $j_*\beta$. Therefore, $x$ can be lifted to $\hat{\alpha}\in\pi_{t-s,w}^{\R}\left(\overline{H_{\R}}^{\wedge s}\right)^h$. Since $r_*\hat{\alpha}$ and $j_*\hat{\beta}$ both map to $j_*\beta$, their difference lies in the image of $\pi_{t-s,w}^{\R}\left(\overline{H_{\R}}^{\wedge (s+1)}\right)^h$. Therefore, $j_*\hat{z}$ is also detected by $x'$.
\end{proof}

\section{Relations to the $\rho$-Bockstein spectral sequence for genuine $C_2$ $\textup{Ext}$}\label{sec-names}
The following are two charts of the $\Ext$ groups in coweight $3$. The first one is $\Ext_{\A_{*,*}^{\R}}^{*,*,*}\left((H_{\R})_{*,*}^h\right)$, where the elements are named from the algebraic Atiyah-Hirzebruch spectral sequences
\begin{equation}\label{w>0}
\bigoplus_{k\ge-w,k\ne-1}\Ext_{\A_*}^{*,*}\left(H_*S^k\right)\Rightarrow \Ext_{\A_*}^{*,*}\left(H_*\Sigma\PP_{-\infty}^{-w-1}\right),
\end{equation}
\begin{equation}\label{w<=0}
\bigoplus_{k\ge-w}\Ext_{\A_*}^{*,*}\left(H_*S^k\right)\Rightarrow \Ext_{\A_*}^{*,*}\left(H_*\Sigma\PP_{-w}^{\infty}\right),
\end{equation}
and the classical Adams $E_2$-term
\begin{equation}\label{-1}
\Ext_{\A_*}^{*,*}\left(H_*S^{-1}\right)
\end{equation}
(c.f. Theorem \ref{posw} and Theorem \ref{negw}). The author produced the chart through a computer program \cite{prog} which computes these spectral sequences through the Curtis algorithm. The second one is the chart in \cite{gi} for $\Ext_{\A_{*,*}^{\R}}^{*,*,*}\left((H_{\R})_{*,*}^{C_2}\right)$, where the elements are named from the $\rho$-Bockstein spectral sequence. Complete charts of $\Ext_{\A_{*,*}^{\R}}^{*,*,*}\left((H_{\R})_{*,*}^h\right)$ produced using these computer calculations for stems $0$ through $30$ and coweights $-2$ through $13$ appear in \cite{charts}.

It is hard not to notice that the two names of elements are highly related, which will be discussed in this section.

\begin{center}
\includegraphics[width=\textwidth]{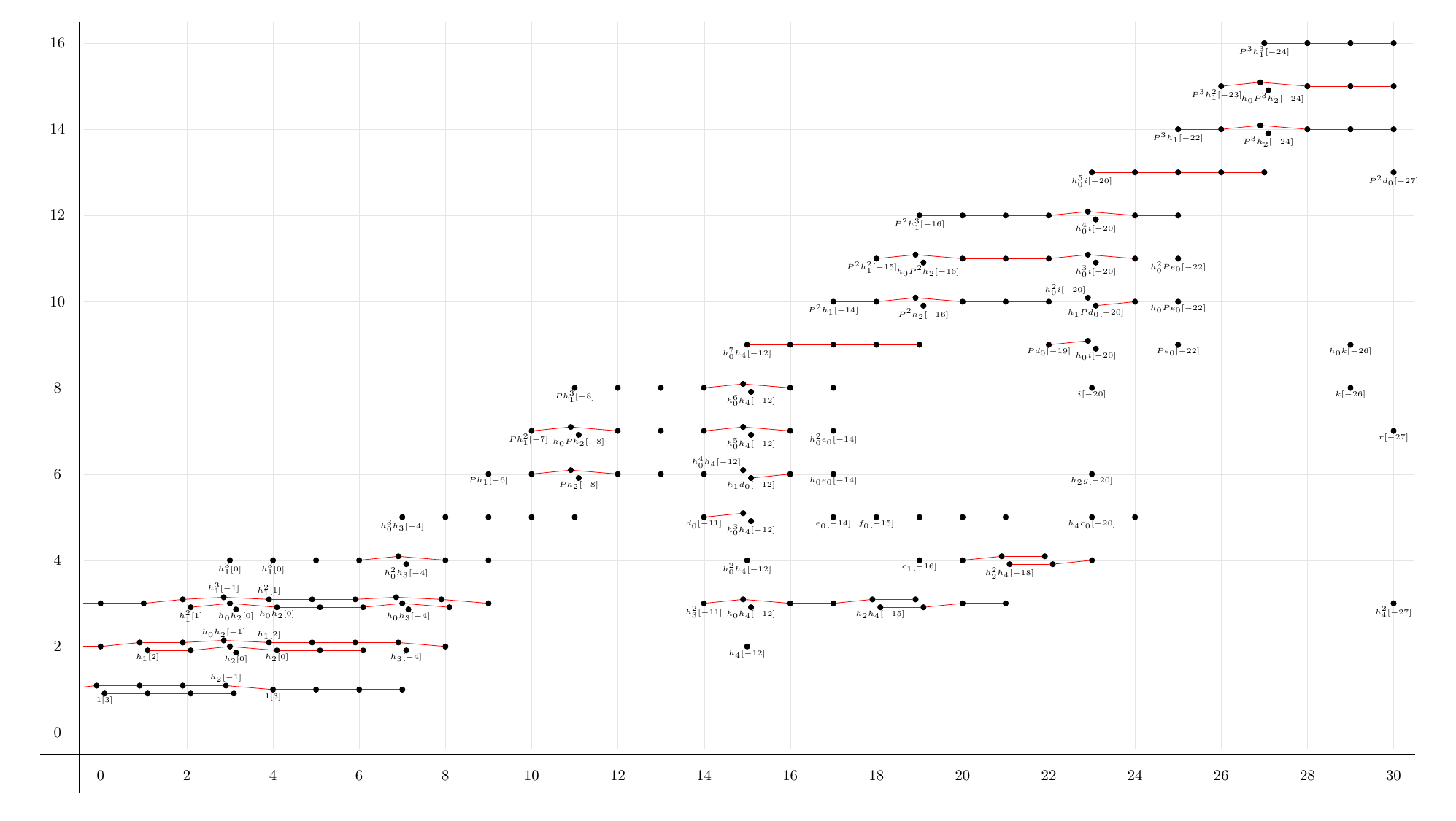}
\[\Ext_{\A_{*,*}^{\R}}^{*,*,*}\left((H_{\R})_{*,*}^h\right)\textup{ in coweight 3}\]
\includegraphics[width=\textwidth]{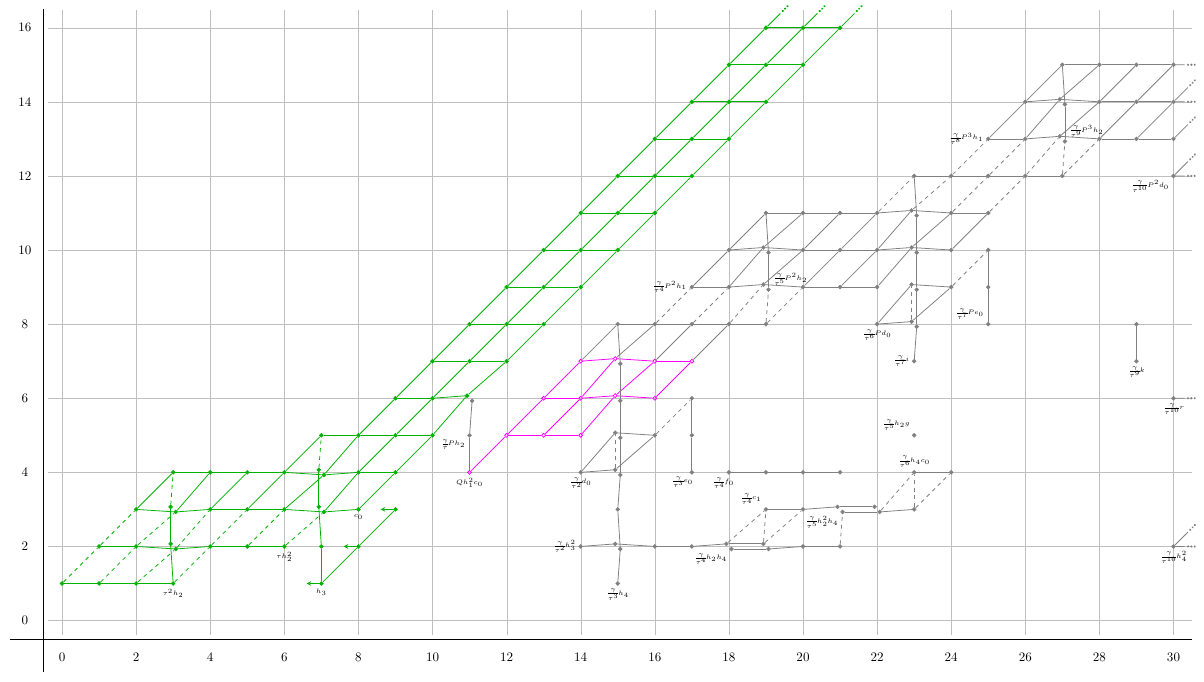}
\[\textup{Guillou-Isaksen's chart (\cite{gi}) of }\Ext_{\A_{*,*}^{\R}}^{*,*,*}\left((H_{\R})_{*,*}^{C_2}\right)\textup{ in coweight 3}\]
\end{center}

By \cite[Prop.~3.1]{GHIR}, the $\rho$-Bockstein spectral sequence that converges to $\Ext_{\A_{*,*}^{\R}}^{*,*,*}\left((H_{\R})_{*,*}^{C_2}\right)$ splits as the direct sum of two $\rho$-Bockstein spectral sequences, one converging to $\Ext_{\A_{*,*}^{\R}}^{*,*,*}\left((H_{\R})_{*,*}\right)$ with the $E_1$-term
\begin{equation}\label{rhopc}
\bigoplus_{j\ge 0}\Ext_{\C}\left\{\rho^j\right\},
\end{equation}
and the other converging to $\Ext_{\A_{*,*}^{\R}}^{*,*,*}\left(\left(H_{\R}\wedge\left(\SR\right)^{\Psi}\right)_{*,*}\right)$ with the $E_1$-term
\begin{equation}\label{rhoncs}
\left(\bigoplus_{j\ge 0}\frac{\F[\tau]}{\left(\tau^{\infty}\right)}\left\{\frac{\gamma}{\rho^j}\right\}\otimes_{\F[\tau]}\Ext_{\C}\right)\oplus\left(\bigoplus_{j\ge 0}\textup{Tor}^{\F[\tau]}\left(\frac{\F[\tau]}{\left(\tau^{\infty}\right)}\left\{\frac{\gamma}{\rho^j}\right\},\Ext_{\C}\right)\right).
\end{equation}

Recall from Section \ref{sec-background} that in the short exact sequence
\begin{equation}\label{ses}
0\to\left(H_{\R}\right)_{*,*}^{\Psi}\to\left(H_{\R}\right)_{*,*}^{\Theta}\to\Sigma^{1,0}\left(H_{\R}\wedge\left(\SR\right)^{\Psi}\right)_{*,*}\to0,
\end{equation}
$\gamma/\tau\in\Sigma^{1,0}\left(H_{\R}\wedge\left(\SR\right)^{\Psi}\right)_{*,*}$ is the image of $1/\left(\tau\rho\right)\in\left(H_{\R}\right)_{*,*}^{\Theta}$. Hence, we can apply a degree shifting to (\ref{rhoncs}) that takes $\gamma/\tau$ to $1/\left(\tau\rho\right)$, and it can be rewritten as the $\rho$-Bockstein spectral sequence converging to $\Ext_{\A_{*,*}^{\R}}^{*,*,*}\left(\Sigma^{1,0}\left(H_{\R}\wedge\left(\SR\right)^{\Psi}\right)_{*,*}\right)$ with the $E_1$-term
\begin{equation}\label{rhonc}
\left(\bigoplus_{j<0}\frac{\F[\tau]}{\left(\tau^{\infty}\right)}\left\{\rho^j\right\}\otimes_{\F[\tau]}\Ext_{\C}\right)\oplus\left(\bigoplus_{j<0}\textup{Tor}^{\F[\tau]}\left(\frac{\F[\tau]}{\left(\tau^{\infty}\right)}\left\{\rho^j\right\},\Ext_{\C}\right)\right).
\end{equation}
There are also the $\rho$-Bockstein spectral sequences converging to $\Ext_{\A_{*,*}^{\R}}^{*,*,*}\left((H_{\R})_{*,*}^\Psi\right)$ and $\Ext_{\A_{*,*}^{\R}}^{*,*,*}\left((H_{\R})_{*,*}^\Theta\right)$, with the $E_1$-terms
\begin{equation}\label{rhopsi}
\left(\bigoplus_{j\ge 0}\frac{\F[\tau]}{\left(\tau^{\infty}\right)}\left\{\rho^j\right\}\otimes_{\F[\tau]}\Ext_{\C}\right)\oplus\left(\bigoplus_{j\ge 0}\textup{Tor}^{\F[\tau]}\left(\frac{\F[\tau]}{\left(\tau^{\infty}\right)}\left\{\rho^j\right\},\Ext_{\C}\right)\right)
\end{equation}
and
\begin{equation}\label{rhotheta}
\left(\bigoplus_{j\in\Z}\frac{\F[\tau]}{\left(\tau^{\infty}\right)}\left\{\rho^j\right\}\otimes_{\F[\tau]}\Ext_{\C}\right)\oplus\left(\bigoplus_{j\in\Z}\textup{Tor}^{\F[\tau]}\left(\frac{\F[\tau]}{\left(\tau^{\infty}\right)}\left\{\rho^j\right\},\Ext_{\C}\right)\right),
\end{equation}
respectively. Note that there are maps between these $\rho$-Bockstein spectral sequences induced by the maps in the short exact sequence (\ref{ses}). Since $\F[\tau^{\pm}]$ is flat over $\F[\tau]$, there is the $\rho$-Bockstein spectral sequence converging to $\Ext_{\A_{*,*}^{\R}}^{*,*,*}\left((H_{\R})^h_{*,*}\right)$ with the $E_1$-term
\begin{equation}\label{rhoh}
\bigoplus_{j\ge 0}\F[\tau^{\pm}]\left\{\rho^j\right\}\otimes_{\F[\tau]}\Ext_{\C}.
\end{equation}

To state our theorems in names comparison, we need to define the names of elements in the $\rho$-Bockstein spectral sequences and the Atiyah-Hirzebruch spectral sequences. For the $\rho$-Bockstein spectral sequences (\ref{rhonc}), (\ref{rhopsi}), and (\ref{rhotheta}), we will only discuss the names of elements in their first summands in the rest of this article.

\begin{defn}\label{bname}
For a homogeneous element $x$ in (\ref{rhopc}) or (\ref{rhoh}), or the first summands of (\ref{rhonc}), (\ref{rhopsi}), or (\ref{rhotheta}), its $\rho$-Bockstein name has the form $\tau^i\rho^j\otimes\alpha$, where $\alpha$ is a $\tau$-periodic element in $\Ext_{\C}$ which is not $\tau$-divisible. The classical Adams part of its $\rho$-Bockstein name is the image of $\alpha$ under the map
\[\Ext_{\C}\to\Ext_{\C}[\tau^{-1}]\xrightarrow{\cong}\Ext_{\A_*}^{*,*}\left(\F\right)[\tau^\pm]\to\Ext_{\A_*}^{*,*}\left(\F\right),\]
where the second map is an isomorphism by \cite[Prop.~3.5]{dimot}, and the third map is the augmentation map.
\end{defn}

\begin{defn}\label{ahname}
Let $J$ be a nonempty subset of $\Z$. Let $M_*$ be an $\A_*$-comodule whose underlying graded $\F$-vector space is
\[M_*=\begin{cases}\F, & *\in J;\\ 0, & *\not\in J.\end{cases}\]
Consider the algebraic Atiyah-Hirzebruch spectral sequence
\[E_1=\bigoplus_{k\in J} \Ext_{\A_*}^{*,*}\left(H_*S^k\right)\Rightarrow \Ext_{A_*}^{*,*}\left(M_*\right).\]
For a homogeneous element $x$ in its $E_1$-page, its Atiyah-Hirzebruch name has the form $\alpha[k]$, where $\alpha\in\Ext_{\A_*}\left(\F\right)$ and $k\in J$. The classical Adams part of its Atiyah-Hirzebruch name is $\alpha$. 

For an element in (\ref{-1}), its Atiyah-Hirzebruch name has the form $\alpha[-1]$, where $\alpha\in\Ext_{\A_*}\left(\F\right)$. The classical Adams part of its Atiyah-Hirzebruch name is $\alpha$.

For an element $x\in\Ext_{\A_{*,*}^{\R}}^{*,*,*}\left((H_{\R})^h_{*,*}\right)$, let $y$ be the element in the $E_2$-terms of (\ref{ass>0}) or (\ref{ass<=0}) corresponding to $x$ under the isomorphisms in Theorem \ref{rhobss}, Theorem \ref{posw}, and Theorem \ref{negw}. The Atiyah-Hirzebruch name of $x$ is the Atiyah-Hirzebruch name of the element detecting $y$ in the algebraic Atiyah-Hirzebruch spectral sequences.
\end{defn}

In order to state the case (\ref{3}) of Theorem \ref{pcnames}, we still need a definition for the Mahowald invariant on the level of $\Ext$ groups.

\begin{defn}\label{root}
Let $\alpha$ be an element in $\Ext_{\A_*}^{*,*}\left(\F\right)$. Let $k$ be the minimal integer such that the composition map
\[\Ext_{\A_*}^{*,*}\left(\F\right)\xrightarrow{\cong}\Ext_{\A_*}^{*,*+1}\left(\varprojlim_kH_*\PP_{-k}^\infty\right)\to\Ext_{\A_*}^{*,*+1}\left(H_*\PP_{-k}^\infty\right)\]
maps $\alpha$ to a nontrivial element $\beta$, where the first map is an isomorphism by Lin's theorem \cite{ldma}. The Mahowald invariant of $\alpha$ is the coset of lifts of $\beta$ along the map
\[\Ext_{\A_*}^{*,*+k+1}\left(\F\right)\cong\Ext_{\A_*}^{*,*+1}\left(H_*S^{-k}\right)\to\Ext_{\A_*}^{*,*+1}\left(H_*\PP_{-k}^\infty\right).\]
\end{defn}

To relate the $\rho$-Bockstein and the Atiyah-Hirzebruch name, we will first refer to Theorem \ref{rhobss} which gives an isomorphism between the algebraic $a$-Bockstein spectral sequence (\ref{abss}) and the algebraic Atiyah-Hirzebruch spectral sequence (\ref{aahss}). Note that under the identification of $\left(Sp^{C_2}\right)^{\wedge}_2$ with a localization of $\left(\Rcell\right)^\wedge_2$ (c.f. Section \ref{sec-background}), the algebraic $a$-Bockstein spectral sequence corresponds to the $\rho$-Bockstein spectral sequence (\ref{rhoh}).

\begin{lem}\label{bah}
Let $x$ be an element in the $E_1$-term of the spectral sequence (\ref{aahss}) corresponding to $y$ in (\ref{rhoh}) under the isomorphism in Theorem \ref{rhobss}. Then the Atiyah-Hirzeburch name of $x$ and the $\rho$-Bockstein name of $y$ have the same classical Adams part.
\end{lem}

\begin{proof}
Note that the isomorphism of the $E_1$-terms of (\ref{aahss}) and (\ref{rhoh})
\begin{equation}\label{rhoah}
\bigoplus_{j\ge 0}\F[\tau^{\pm}]\left\{\rho^j\right\}\otimes_{\F[\tau]}\Ext_{\C}\cong\bigoplus_w\bigoplus_{k=-\infty}^{-w-1}\Ext_{\A_*}\left(H_*S^k\right)
\end{equation}
can be written as the composition of isomorphisms:
\begin{equation}\label{rhoahiso}
\begin{split}
\bigoplus_{j\ge 0}\F[\tau^{\pm}]\left\{\rho^j\right\}\otimes_{\F[\tau]}\Ext_{\C} & \cong \bigoplus_{j\ge 0}\Ext_{\C}[\tau^{-1}]\left\{\rho^j\right\} \\
& \cong \bigoplus_{j\ge 0}\Ext_{\A_*}\left(\F\right)[\tau^\pm]\left\{\rho^j\right\}\\
& \cong \bigoplus_w\bigoplus_{k=-\infty}^{-w-1}\Ext_{\A_*}\left(H_*S^k\right).
\end{split}
\end{equation}
Then the result follows from Definition \ref{bname} and Definition \ref{ahname}.
\end{proof}

This lemma connects the $\rho$-Bockstein names of elements in (\ref{rhoh}) with the Atiyah-Hirzebruch names of elements in (\ref{aahss}). We still need to compare the Atiyah-Hirzebruch names of elements in (\ref{aahss}), (\ref{w>0}), and (\ref{w<=0}), and the $\rho$-Bockstein names of elements in (\ref{rhoh}), (\ref{rhopc}), and (\ref{rhonc}).
\begin{equation}\label{key}
\begin{aligned}
\xymatrix@C=1in{
*+[F]{\txt{$\rho$-Bockstein\\names in $E_2^{C_2}$}} \ar@{<-->}[r] \ar@{<->}[d]_{(\dagger)} & *+[F]{\txt{Atiyah-Hirzebruch names in\\(\ref{w>0}), (\ref{w<=0}), and (\ref{-1})}} \ar@{<->}[d]^{(\dagger\dagger)} \\
*+[F]{\txt{$\rho$-Bockstein\\names in $E_2^h$}} \ar@{<->}[r]_{\txt{Lemma \ref{bah}}}  & *+[F]{\txt{Atiyah-Hirzebruch\\names in (\ref{aahss})}}
}
\end{aligned}
\end{equation}

We first consider elements in the positive cone part of $E_2^{C_2}$, in which case ($\dagger$) is induced by the map
\begin{equation}\label{pch}
\Ext_{\A_{*,*}^{\R}}^{*,*,*}\left((H_{\R})_{*,*}\right)\to\Ext_{\A_{*,*}^{\R}}^{*,*,*}\left((H_{\R})^h_{*,*}\right).
\end{equation}
Let $x$ be an element in (\ref{rhopc}) that survives to
\[[x]\in\Ext_{\A_{*,*}^{\R}}^{*,*,*}\left(\left(H_{\R}\right)_{*,*}\right)\subset E_2^{C_2}.\]
Take an element $y$ in (\ref{rhoh}) such that it has the same $\rho$-Bockstein name with $x$. Suppose that $y$ survives to $[y]\in\Ext_{\A_{*,*}^{\R}}^{s,t,w}\left((H_{\R})^h_{*,*}\right)$. Then $[y]$ is the image of $[x]$ under the map (\ref{pch}). The following theorem characterizes the behavior of ($\dagger\dagger$).

\begin{thm}\label{pcnames}
Let $x$ and $y$ be as above. Then
\begin{enumerate}
\item\label{1} If $w>0$, there is a differential in the algebraic Atiyah-Hirzebruch spectral sequence for $\Ext_{\A_*}(\varprojlim H_*\Sigma\PP_{-\infty}^{-k-1})$ whose source and target have the same classical Adams parts as the Atiyah-Hirzebruch name of $[y]$ and the $\rho$-Bockstein name of $x$ respectively.
\item If $w\le0$ and $[x]$ is a $\rho$-torsion, there is a differential in the algebraic Atiyah-Hirzebruch spectral sequence for $\Ext_{\A_*}^{*,*}(\varprojlim H_*\PP_{-k}^{\infty})$ whose source and target have the same classical Adams parts as the Atiyah-Hirzebruch name of $[y]$ and the $\rho$-Bockstein name of $x$ respectively.
\item\label{3} If $w\le0$ and $[x]$ is $\rho$-periodic, the classical Adams part of the $\rho$-Bockstein name of $x$ is a Mahowald invariant of the classical Adams part of the Atiyah-Hirzebruch name of $[y]$.
\end{enumerate}
\end{thm}

\begin{proof}
Let $y'$ be the image of $y$ under the isomorphism (\ref{rhoah}). Then the $\rho$-Bockstein name of $x$ is the same as the $\rho$-Bockstein name of $y$, which has the same classical Adams parts as the Atiyah-Hirzebruch name of $y'$ by Lemma \ref{bah}.
\begin{enumerate}
\item When $w>0$, recall from the proof of Theorem \ref{posw} that the connecting homomorphism
\begin{equation}\label{extlim}
\Ext_{\A_*}^{*,*}\left(H_*\Sigma\PP_{-\infty}^{-w-1}\right)\to \Ext_{\A_*}^{*+1,*}\left(\varprojlim_kH_*\PP_{-k}^{-w-1}\right)
\end{equation}
induced by the short exact sequence
\[0\to\varprojlim\limits_kH_*\PP_{-k}^{-w-1}\to\varprojlim\limits_kH_*\Sigma\PP_{-\infty}^{-k-1}\to H_*\Sigma\PP_{-\infty}^{-w-1}\to0\]
is an isomorphism. Let $[z]$ be the preimage of $[y']$ along the map (\ref{extlim}), which is detected by $z$ in the $E_1$-term of (\ref{w>0}). By Definition \ref{ahname}, the Atiyah-Hirzebruch name of $[y]$ is the Atiyah-Hirzebruch name of $z$. Now we need to compare the Atiyah-Hirzebruch names of $y'$ and $z$. Unpacking the connecting homomorphism to the chain level, we can see that there is an element representing $z$ in the cobar complex $C_{\A_*}(\varprojlim H_*\Sigma\PP_{-\infty}^{-k-1})$ whose chain differential has the leading term representing $y'$. Then the result follows since the differentials in the algebraic Atiyah-Hirzebruch spectral sequence come from the chain differentials.
\item When $w\ge0$, recall from the proof of Theorem \ref{negw} that there is a short exact sequence
\begin{equation}\label{extses}
\begin{split}
0\to\Ext_{\A_*}^{*,*}\left(H_*\PP_{-w}^{\infty}\right)&\to\Ext_{\A_*}^{*+1,*}\left(\varprojlim\limits_kH_*\PP_{-k}^{-w-1}\right)\\
&\to\Ext_{\A_*}^{*+1,*}\left(\varprojlim\limits_kH_*\PP_{-k}^{\infty}\right)\to0.
\end{split}
\end{equation}
By Theorem \ref{rhobss}, the assumption that $[x]$ is a $\rho$-torsion implies that $[y']$ maps to $0$ in $\Ext_{\A_*}^{*,*}(\varprojlim H_*\PP_{-k}^{\infty})$, and thus can be lifted to $[z]\in\Ext_{\A_*}^{*,*}(H_*\PP_{-w}^{\infty})$ detected by $z$ in the $E_1$-term of (\ref{w<=0}). Then the Atiyah-Hirzebruch name of $[y]$ is the Atiyah-Hirzebruch name of $z$. Note that the injective map in (\ref{extses}) is the connecting homomorphism induced by the short exact sequence
\[0\to\varprojlim\limits_kH_*\PP_{-k}^{-w-1}\to\varprojlim\limits_kH_*\PP_{-k}^{\infty}\to H_*\PP_{-w}^{\infty}\to0.\]
Then the result follows by the same method as in the proof of (\ref{1}).
\item In this case, the Atiyah-Hirzebruch name of $[y]$ is the image of $[y']$ under the composition map
\[\Ext_{\A_*}^{*,*}\left(\varprojlim\limits_kH_*\PP_{-k}^{-w-1}\right)\to\Ext_{\A_*}^{*,*}\left(\varprojlim\limits_kH_*\PP_{-k}^{\infty}\right)\xrightarrow{\cong}\Ext_{\A_*}^{*,*+1}\left(\F\right),\]
where the first map is the surjective map in (\ref{extses}), and the second map is contructed from Lin's theorem \cite{ldma}. Then the result follows from Definition \ref{root}. \qedhere
\end{enumerate}
\end{proof}

\begin{rmk}
Theorem \ref{pcnames} shows that for the elements in $\Ext_{\A_{*,*}^{\R}}^{*,*,*}\left((H_{\R})^h_{*,*}\right)$ that differ by a $\rho$-multiplication, their Atiyah-Hirzebruch names can still be different if the weight of one element is positive and the weight of the other is non-positive. For example, there is a differential $d(h_0h_2[0])=c_0[-6]$ in the algebraic Atiyah-Hirzebruch spectral sequence for $\Ext_{\A_*}(\varprojlim H_*\Sigma\PP_{-\infty}^{-k-1})$\footnote{This differential can be read off from the output of the author's computer program \cite{prog}, so can other differentials in the algebraic Atiyah-Hirzebruch spectral sequence appeared in the remarks in this section.}, so that $\rho^4c_0\in\Ext_{\A_{*,*}^{\R}}^{3,7,1}\left((H_{\R})_{*,*}\right)$ maps to $h_0h_2[0]\in\Ext_{\A_{*,*}^{\R}}^{3,7,1}\left((H_{\R})^h_{*,*}\right)$. There is a differential $d(h_1^2[1])=c_0[-6]$ in the algebraic Atiyah-Hirzebruch spectral sequence for $\Ext_{\A_*}^{*,*}(\varprojlim H_*\PP_{-k}^{\infty})$, so that $\rho^5c_0\in\Ext_{\A_{*,*}^{\R}}^{3,6,0}\left((H_{\R})_{*,*}\right)$ maps to $h_1^2[1]\in\Ext_{\A_{*,*}^{\R}}^{3,6,0}\left((H_{\R})^h_{*,*}\right)$. The Atiyah-Hirzebruch names of $h_0h_2[0]$ and $h_1^2[1]$ have different classical Adams parts, but they only differ by a $\rho$-multiplication.
\end{rmk}

\begin{rmk}
If $y$ does not survive to the $E_{\infty}$-page of (\ref{rhoh}), there will be two cases. If the image of $[x]$ is trivial in $\Ext_{\A_{*,*}^{\R}}^{*,*,*}\left((H_{\R})^h_{*,*}\right)$, then $[x]$ is killed by the shortened differential $\delta$ (c.f. Theorem \ref{borext}), which detects a nontrivial $d_2$ in the genuine $C_2$-equivariant Adams spectral sequence. If its image is nontrivial, then there is a hidden $\rho$-extension in the genuine $C_2$-equivariant Adams spectral sequence detected by the $E_2$-term of the Borel $C_2$-equivariant Adams spectral sequence. For example, $h_1^2c_0$ and $h_1^6c_0$ do not survive to the $E_1$-page of the $\rho$-Bockstein spectral sequence for $\Ext_{\A_{*,*}^{\R}}^{*,*,*}\left((H_{\R})^h_{*,*}\right)$. The image of $h_1^6c_0\in\Ext_{\A_{*,*}^{\R}}^{9,23,11}\left((H_{\R})_{*,*}\right)$ is trivial in $\Ext_{\A_{*,*}^{\R}}^{9,23,11}\left((H_{\R})^h_{*,*}\right)$, which implies that it is killed by $\delta$, whose source turns out to be $\rho^{-2}QPh_1^4\in\Ext_{\A_{*,*}^{\R}}^{7,22,11}\left(\left(H_{\R}\wedge\left(\SR\right)^{\Psi}\right)_{*,*}\right)$ by the degree reason. On the other hand, the image of $h_1^2c_0\in\Ext_{\A_{*,*}^{\R}}^{5,15,7}\left((H_{\R})_{*,*}\right)$ is nontrivial in $\Ext_{\A_{*,*}^{\R}}^{5,15,7}\left((H_{\R})^h_{*,*}\right)$. This implies that the homotopy class detected by $h_1^2c_0$ has to be $\rho$-divisible, which could also be read from the charts of $\Ext_{\A_{*,*}^{\R}}^{*,*,*}\left((H_{\R})^h_{*,*}\right)$.
\end{rmk}

Then we consider elements in the negative cone part of $E_2^{C_2}$, in which case ($\dagger$) is induced by the zig-zag of maps as shown in the proof of Theorem \ref{borext}:
\[\Ext_{\A_{*,*}^{\R}}^{*,*,*}\left(\Sigma^{1,0}\left(H_{\R}\wedge\left(\SR\right)^{\Psi}\right)_{*,*}\right)\xrightarrow{\cong}\Ext_{\A_{*,*}^{\R}}^{*+1,*,*}\left((H_{\R})^\Psi_{*,*}\right)\leftarrow\Ext_{\A_{*,*}^{\R}}^{*+1,*,*}\left((H_{\R})^h_{*,*}\right).\]
Let $x$ be an element in the first summand of (\ref{rhonc}) that survives to 
\[[x]\in\Ext_{\A_{*,*}^{\R}}^{*,*,*}\left(\Sigma^{1,0}\left(H_{\R}\wedge\left(\SR\right)^{\Psi}\right)_{*,*}\right)\subset E_2^{C_2}.\]
To have $[x]$ survive in $\Ext_{\A_{*,*}^{\R}}^{*,*,*}\left((H_{\R})^h_{*,*}\right)$, we need to assume that $\delta([x])=0$, where $\delta$ is the map constructed in Theorem \ref{borext}. Take an element in the first summand of (\ref{rhopsi}) such that it has the same $\rho$-Bockstein name with $x$, which, by abuse of notation, will also be denoted by $x$. Since $\Ext_{\A_{*,*}^{\R}}^{*,*,*}\left((H_{\R})^{\Theta}_{*,*}\right)=0$, no elements survive in the $\rho$-Bockstein spectral sequence (\ref{rhotheta}). Because $x$ survives in (\ref{rhonc}), it cannot be a target of differentials in (\ref{rhotheta}), and thus it supports a differential $d(x)=y$. Note that $x$ is a permanent cycle in (\ref{rhonc}), hence $y$ belongs to (\ref{rhopsi}).

Now we claim that $y$ can be chosen so that it belongs to the first summand of (\ref{rhopsi}). Note that every element in $\F[\tau]/\left(\tau^\infty\right)$ is infinitely $\tau$-divisible. In particular, $x$ is $\tau^{2^k}$-divisible for any positive integer $k$. Let $x'$ be an element in the first summand of (\ref{rhonc}) such that $x=\tau^{2^k}x'$. In the proof of \cite[Lem.~4.3]{di}, Dugger-Isaksen showed that $\tau^{2^k}$ is a cycle in the $E_{2^k-1}$-page of the $\rho$-Bockstein spectral sequence, so $d(x)=\tau^{2^k}d(x')$ for $k\gg0$. This implies that $y$ can be chosen to be $\tau^{2^k}$-divisible for arbitrary $k\gg0$, and the claim follows.

Since $y=d(x)$ is a permanent cycle in the spectral sequence (\ref{rhotheta}), it is also a permanent cycle in the spectral sequence (\ref{rhopsi}). Furthermore, it is not a target of a differential in (\ref{rhopsi}), otherwise it would be the target of a differential in (\ref{rhotheta}) which happens before $d(x)=y$. Therefore, it survives to $[y]\in\Ext_{\A_{*,*}^{\R}}^{*,*,*}\left((H_{\R})^\Psi_{*,*}\right)$, which is the image of $[x]$ under the connecting homomorphism
\[\Ext_{\A_{*,*}^{\R}}^{*,*,*}\left(\Sigma^{1,0}\left(H_{\R}\wedge\left(\SR\right)^{\Psi}\right)_{*,*}\right)\to\Ext_{\A_{*,*}^{\R}}^{*+1,*,*}\left((H_{\R})^\Psi_{*,*}\right).\]
In the proof of Theorem \ref{borext}, we showed that $\delta$ coincides with the map $\delta'$ constructed in (\ref{delta}). Then our assumption $\delta([x])=0$ implies that $[y]$ maps to $0$ under the map
\[\Ext_{\A_{*,*}^{\R}}^{*,*,*}\left((H_{\R})^\Psi_{*,*}\right)\to\Ext_{\A_{*,*}^{\R}}^{*+1,*,*}\left((H_{\R})_{*,*}\right).\]

Let $\overline{y}$ be an element in (\ref{rhoh}) which has the same $\rho$-Bockstein name as $y$. Then $\overline{y}$ survives to $[\overline{y}]\in\Ext_{\A_{*,*}^{\R}}^{*+1,*,*}\left((H_{\R})^h_{*,*}\right)$ in the spectral sequence (\ref{rhoh}). Since $[\overline{y}]$ maps to $[y]$ under the map
\[\Ext_{\A_{*,*}^{\R}}^{*,*,*}\left((H_{\R})^h_{*,*}\right)\to\Ext_{\A_{*,*}^{\R}}^{*,*,*}\left((H_{\R})^\Psi_{*,*}\right),\]
we can see that $[\overline{y}]$ corresponds to $[x]$ under the identification in Theorem \ref{borext}. This gives a way to characterize ($\dagger$) in (\ref{key}).

Take the minimal $i>0$ such that $\rho^i x$ lies in the first summand of (\ref{rhopsi}). Then there is a differential $d(\rho^ix)=\rho^iy$ in the $\rho$-Bockstein spectral sequence (\ref{rhopsi}).

\begin{thm}\label{ncnames}
Let $x$, $y$, $\overline{y}$, and $i$ be as above. Let $\overline{\rho^ix}$ be an element in (\ref{rhoh}) with the same $\rho$-Bockstein name as $\rho^ix$. Suppose there is a differential $d(\overline{\rho^ix})=\rho^i\overline{y}$ in the $\rho$-Bockstein spectral sequence (\ref{rhoh}). Then the $\rho$-Bockstein name of $x$ and the Atiyah-Hirzebruch name of $[\overline{y}]$ have the same classical Adams part.
\end{thm}

\begin{proof}
Since $i$ is minimal, the $\rho$-Bockstein name of $\rho^ix\in\F[\tau]/(\tau^\infty)\otimes_{\F[\tau]}\Ext_{\C}$ is $\tau^m\otimes\alpha$, where $m<0$ and $\alpha$ is not $\tau$-divisible. By \cite[Prop.~2.2(1)]{bgi}, the weight of $\alpha$ is non-negative. Since the weight of $\tau^m$ is positive, the weight of $\rho^ix$ is also positive, and so are the weights of $\overline{\rho^ix}$ and $\rho^i\overline{y}$. Let $w$ be the weight of $\overline{y}$. Since the weight of $\rho^i$ is $-i$, the weight of $\rho^i\overline{y}$ is $w-i$, and we have $w>i\ge1$.

Let $\tilde{x}$ and $\tilde{y}$ be the images of $\overline{\rho^ix}$ and $\rho^i\overline{y}$ respectively under the isomorphism (\ref{rhoah}). By Theorem \ref{rhobss}, there is a differential $d(\tilde{x})=\tilde{y}$ in the algebraic Atiyah-Hirzebruch spectral sequence
\[\bigoplus_{l=-\infty}^{-(w-i)-1}\Ext_{\A_*}^{*,*}\left(H_*S^l\right)\Rightarrow \Ext_{\A_*}^{*,*}\left(\varprojlim_kH_*\PP_{-k}^{-(w-i)-1}\right).\]
Note that the inclusion of $\A_*$-comodules
\[\varprojlim_kH_*\PP_{-k}^{-(w-i)-1}\to\varprojlim_kH_*\Sigma\PP_{-\infty}^{-k-1}\]
is an isomorphism when $*\le-(w-i)-1$. Therefore, the differential $d(\tilde{x})=\tilde{y}$ also occurs in the algebraic Atiyah-Hirzebruch spectral sequence
\begin{equation}\label{ah-1}
\bigoplus_{l\ne-1}\Ext_{\A_*}^{*,*}\left(H_*S^l\right)\Rightarrow \Ext_{\A_*}^{*,*}\left(\varprojlim_kH_*\Sigma\PP_{-\infty}^{-k-1}\right).
\end{equation}

By (\ref{rhoahiso}), the image of $\overline{y}$ under the isomorphism (\ref{rhoah}) is also $\tilde{y}$. Since $\overline{y}$ is assumed to survive to $[\overline{y}]\in\Ext_{\A_{*,*}^{\R}}^{*,*,*}\left((H_{\R})^h_{*,*}\right)$, by Theorem \ref{rhobss}, $\tilde{y}$ survives to $[\tilde{y}]\in\Ext_{\A_*}^{*,*}\left(\varprojlim H_*\PP_{-k}^{-w-1}\right)$. Using the same method as in the proof of Theorem \ref{pcnames} (\ref{1}), we can see that $\tilde{x}$ survives to $[\tilde{x}]\in\Ext_{\A_*}^{*-1,*}\left(H_*\Sigma\PP_{-\infty}^{-w-1}\right)$, which maps to $[\tilde{y}]$ under the isomorphism (\ref{extlim}).

By Definition \ref{ahname}, the Atiyah-Hirzebruch name of $[\overline{y}]$ is the Atiyah-Hirzebruch name of $\tilde{x}$. By Lemma \ref{bah}, the Atiyah-Hirzebruch name of $\tilde{x}$ and the $\rho$-Bockstein name of $\overline{\rho^ix}$ have the same classical Adams part. It is direct that the $\rho$-Bockstein name of $\overline{\rho^ix}$ and the $\rho$-Bockstein name of $x$ have the same classical Adams part. Therefore, the result follows.
\end{proof}

\begin{rmk}
The assumption that there is a differential $d(\overline{\rho^ix})=\rho^i\overline{y}$ might fail. This is because there might be $d(\overline{\rho^ix})=\rho^i\overline{y}+z$, where $z$ is an element in (\ref{rhopc}) with lower $\rho$-power, which leads to the situation that $\overline{\rho^i x}$ supports a shorter differential. However, this will not happen if every $\rho$-Bockstein differential in (\ref{rhoh}) does not increase the power of $\tau$ in the $\rho$-Bockstein names. Since $\tau^{2^k}$ is a cycle in the $E_{2^k-1}$-page (see the proof of \cite[Lem.~4.3]{di}), this is equivalent to say that every differential in (\ref{rhopc}) does not increase the $\tau$-divisibility. Unfortuantely, we did not figure out a proof for that, neither have we found a counterexample. But still, we are optimistic that the following conjecture holds.
\end{rmk}

\begin{conj}\label{conjnc}
Let $d(A)=B$ be a differential in the $\rho$-Bockstein spectral sequence
\[\bigoplus_{j\ge 0}\Ext_{\C}\left\{\rho^j\right\}\Rightarrow\Ext_{\A_{*,*}^{\R}}^{*,*,*}\left((H_{\R})_{*,*}\right).\]
Let $\tau^i\rho^j\otimes\alpha$ and $\tau^{i'}\rho^{j'}\otimes\beta$ be the $\rho$-Bockstein names of $A$ and $B$ respectively. Then $i\ge i'$. Consequently, the $\rho$-Bockstein name of $x$ and the Atiyah-Hirzebruch name of $[\overline{y}]$ always have the same classical Adams part.
\end{conj}

\begin{rmk}
For an element in the negative cone part of the $E_2$-term of the genuine $C_2$-equivariant spectral sequence whose representative lies in the second summand of (\ref{rhonc}), the relation between the $\rho$-Bockstein names and the Atiyah-Hirzebruch names is not clear. However, the comparison between the Borel and genuine $C_2$-equivariant spectral sequences will provide us some information in the $\rho$-Bockstein spectral sequence (\ref{rhotheta}), which may be helpful in other computations. For example, $Qh_1^2c_0\in\Ext_{\A_{*,*}^{\R}}^{4,15,8}\left(\left(H_{\R}\wedge\left(\SR\right)^{\Psi}\right)_{*,*}\right)$ can be lifted to $h_0^3h_3[-4]\in\Ext_{\A_{*,*}^{\R}}^{5,16,8}\left((H_{\R})^h_{*,*}\right)$ due to the degree reason. From the differential $d(h_0^3h_3[-4])=Ph_2[-9]$ in (\ref{ah-1}), we can see that $Qh_1^2c_0$ maps to $\tau^{-2}Ph_2$ under the connecting homomorphism
\[\Ext_{\A_{*,*}^{\R}}^{*,*,*}\left(\left(H_{\R}\wedge\left(\SR\right)^{\Psi}\right)_{*,*}\right)\to\Ext_{\A_{*,*}^{\R}}^{*+1,*+1,*}\left((H_{\R})^{\Psi}_{*,*}\right).\]
\end{rmk}

%\appendix

%\section{The charts of $\Ext_{\A_{*,*}^{\R}}^{*,*,*}\left((H_{\R})^h_{*,*}\right)$}\label{sec-charts}
%The results of this article connect Mahowald's computations of the classical Adams spectral sequences \cite{mah} to Guillou-Isaksen's computations of the genuine $C_2$-equivariant Adams spectral sequences \cite{gi} (c.f. also \cite{bi}, \cite{bgi}). These computations go up to the $30$-stem. The connection between the classical Adams spectral sequences of $\PP_k^\infty$ and the genuine $C_2$-equivariant spectral sequence for the sphere relies on a comparison of $\Ext_{\A_*}^{*,*}\left(H_*\PP_k^\infty\right)$, $\Ext_{\A_{*,*}^{\R}}^{*,*,*}\left((H_{\R})^h_{*,*}\right)$, and $\Ext_{\A_{*,*}^{\R}}^{*,*,*}\left((H_{\R})^{C_2}_{*,*}\right)$ (see Section \ref{sec-borel}, \ref{sec-limit}, \ref{sec-genuine}, \ref{sec-names}). The charts in this Appendix give the complete computations of $\Ext_{\A_{*,*}^{\R}}^{*,*,*}\left((H_{\R})^h_{*,*}\right)$ for stems $0$ through $30$ and coweights $-2$ through $13$.

%For legibility, each chart contains the data in a fixed coweight ($t-s-w$). The horizontal coordinate is the topological degree ($t-s$), and the vertical coordinate is the Adams filtration degree ($s$). The red lines in the charts indicate the $\rho$-multiplication. The elements are labeled by their Atiyah-Hirzebruch names (c.f. Section \ref{sec-names}). If two elements are connected by red lines, and their weights are both positive or both non-positive, then they have the same Atiyah-Hirzebruch names.

%\includepdf[pages=-,angle=-90]{charts/Charts.pdf}

\bibliographystyle{alpha}
\bibliography{bor}

\end{document}